\documentclass[12pt]{amsart}
\usepackage{amsmath,amssymb,mathrsfs,bm}
\usepackage{pstricks}
\usepackage[colorlinks=true,linkcolor=black,citecolor=black,urlcolor=black]{hyperref}

\psset{unit=1pt}
\psset{arrowsize=4pt 1}
\psset{linewidth=.5pt}

\newgray{lightgray}{.80}

\countdef\x=23
\countdef\y=24
\countdef\z=25
\countdef\t=26

\def\graybox(#1,#2){
\x=#1 \y=#2 
\z=\x \t=\y
\advance\z by 10 
\advance\t by 10 
\psframe[fillstyle=solid,fillcolor=lightgray,linewidth=0pt](\x,\y)(\z,\t) 
\psline[linewidth=.5pt](\x,\y)(\x,\t)(\z,\t)(\z,\y)(\x,\y)}

\def\emptygraybox(#1,#2){
\x=#1 \y=#2 
\z=\x \t=\y
\advance\z by 10 
\advance\t by 10 
\psframe[fillstyle=solid,fillcolor=lightgray,linewidth=0pt,linecolor=lightgray](\x,\y)(\z,\t)}

\def\blankbox(#1,#2){
\x=#1 \y=#2 
\z=\x \t=\y
\advance\z by 10 
\advance\t by 10 
\psframe[linewidth=.5pt](\x,\y)(\z,\t)}

\def\whitebox(#1,#2){
\x=#1 \y=#2 
\z=\x \t=\y
\advance\z by 10 
\advance\t by 10 
\psframe[fillstyle=solid,fillcolor=white,linewidth=0pt](\x,\y)(\z,\t) 
\psline[linewidth=.5pt](\x,\y)(\x,\t)(\z,\t)(\z,\y)(\x,\y)}

\def\whiteboxb(#1,#2){
\x=#1 \y=#2 
\z=\x \t=\y
\advance\z by 10 
\advance\t by 10 
\psframe[fillstyle=solid,fillcolor=white,linewidth=0pt](\x,\y)(\z,\t)}

\newcommand{\define}{\textbf}

\newcommand{\excise}[1]{}

\renewcommand{\setminus}{\smallsetminus}
\renewcommand{\phi}{\varphi}

\renewcommand{\tilde}{\widetilde}

\renewcommand{\bar}{\overline}

\newcommand{\C}{\mathbb{C}}

\newcommand{\Ess}{\mathscr{E}\hspace{-.4ex}ss}
\newcommand{\RWY}{\mathcal{E}_\mathrm{RWY}}
\newcommand{\Max}{\mathscr{M}}

\newcommand{\ee}{\epsilon}

\DeclareMathOperator{\Null}{nullity}

\newtheorem{theorem}{Theorem}[section]
\newtheorem{lemma}[theorem]{Lemma}
\newtheorem{proposition}[theorem]{Proposition}
\newtheorem{corollary}[theorem]{Corollary}

\newtheorem*{thm*}{Theorem}
\newtheorem*{lem*}{Lemma}
\newtheorem*{prop*}{Proposition}
\newtheorem*{cor*}{Corollary}

\newtheorem*{thmA}{Theorem~A}
\newtheorem*{thmB}{Theorem~B}

\theoremstyle{definition}

\newtheorem{definition}[theorem]{Definition}
\newtheorem{remark}[theorem]{Remark}
\newtheorem{example}[theorem]{Example}

\newtheorem*{defn*}{Definition}
\newtheorem*{rmk*}{Remark}

\begin{document}

\title[Diagrams and essential sets for signed permutations]{Diagrams and essential sets \\ for signed permutations}

\date{July 21, 2016}

\author{David Anderson}
\address{Department of Mathematics, The Ohio State University, Columbus, OH 43210}
\email{anderson.2804@math.osu.edu}

\thanks{This work was partially supported by NSF Grant DMS-1502201 and a postdoctoral fellowship from the Instituto Nacional de Matem\'atica Pura e Aplicada (IMPA)}

\begin{abstract}
We introduce diagrams and essential sets for signed permutations, extending the analogous notions for ordinary permutations.  In particular, we show that the essential set provides a minimal list of rank conditions defining the Schubert variety or degeneracy locus corresponding to a signed permutation.  Our essential set is in bijection with the poset-theoretic version defined by Reiner, Woo, and Yong, and thus gives an explicit, diagrammatic method for computing the latter.
\end{abstract}

\maketitle

\section*{Introduction}

Representing a permutation in $S_n$ as a matrix of dots in an $n\times n$ array of boxes, its {\em (Rothe) diagram} is the subset of boxes that remain after striking out the boxes (weakly) south or east of each dot.  The boxes of the diagram correspond naturally to the inversions of the permutation, so the number of them is equal to the length of the permutation.  In fact, the diagram is a convenient way of encoding a great deal of information about a permutation.  The goal of this article is to develop an analogous tool for signed permutations.

A permutation $v\in S_n$ determines $n^2$ rank conditions on the set of all $n\times n$ matrices, by requiring that each lower-left submatrix have rank at least that of the corresponding part of the permutation matrix for $v$; more generally, $v$ determines a degeneracy locus for a flagged vector bundle on a variety, by imposing this condition on each fiber.  These rank conditions are highly redundant, however.  A minimal list of non-redundant rank conditions is determined by the {\em essential set} of $v$, introduced in \cite{fulton}.  It can be read easily from the diagram of $v$: the definition of the essential set identifies it with the set of southeast corners of the diagram.

Seeking formulas for degeneracy loci of other classical types \cite{af0,af1}, we were motivated to find minimal lists of rank conditions corresponding to {\em signed permutations}, since these index degeneracy loci in other types.  For example, in the type B setting of an odd-rank vector bundle equipped with a nondegenerate quadratic form, a degeneracy locus can be represented (locally) by a $(2n+1)\times n$ matrix whose columns are required to be isotropic and mutually orthogonal.  A signed permutation $w$ specifies $2n^2+n$ rank conditions on such a matrix, most of which are redundant; we shall find a minimal {\em essential set} of rank conditions which suffice, and can be read easily from the {\em diagram} of $w$, which are introduced here.

All the statements and most of the arguments in the body of this paper will be combinatorial, but it will be useful to explain some of this geometric motivation in more detail first.  
Consider an odd-rank vector bundle $V$ on a variety $X$, equipped with a nondegenerate quadratic form and flags of subbundles $V \supset E_1 \supset E_2 \supset \cdots$ and $V\supset F_1 \supset F_2 \supset \cdots$, with $E_1$ and $F_1$ maximal isotropic.  A signed permutation $w$ defines a degeneracy locus $\Omega_w\subseteq X$ by imposing certain rank conditions $\dim(E_p\cap F_q)\geq k$.  Here $k=r_w(p,q)$ is the value of a {\em rank function} associated with $w$.

To describe the rank conditions, it is enough to consider the special case where $X$ is the odd orthogonal flag variety, and $\Omega_w$ is a Schubert variety of codimension equal to the length of $w$.\footnote{In fact, this case is all we need for the purposes of this article; the reader may freely substitute ``Schubert variety'' for ``degeneracy locus'', or vice versa, according to taste.}  Let us assume the flag $F_\bullet$ comes from a standard basis $e_n,\ldots,e_1, e_0, e_{\bar{1}}, \ldots, e_{\bar{n}}$ (using $\bar{q}$ to denote $-q$).  That is, extending to a complete flag by setting $F_0=F_1^\perp$, $F_{\bar{1}}=F_2^\perp$, etc. (so $F_{\bar{q}} = F_{q+1}^\perp$), the space $F_q$ is spanned by $\{e_i\,|\, i\geq q\}$ for all $q$.  An isotropic flag $E_\bullet$ can be represented by a $(2n+1)\times n$ matrix whose rows correspond to the basis vectors $e_i$: labelling the columns $\bar{n},\ldots,\bar{1}$, the space $E_p$ is the span of the columns whose labels are at most $\bar{p}$.  One can always scale the bottom-most nonzero entry in each column to $1$; the remaining entries are either free or determined from the other entries by the isotropicity condition.  If the $1$'s are in positions $(\bar{w(i)},\bar{\imath})$ for $i=1,\ldots,n$, then $\dim(E_p\cap F_q) = \#\{i\leq \bar{p} \,|\, w(i)\geq q\}$, and this defines the rank function $r_w(p,q)$.  The Schubert variety $\Omega_w$ is defined as the set of all $E_\bullet$ such that $\dim(E_p\cap F_q) \geq r_w(p,q)$ for $1\leq p\leq n$ and $\bar{n}\leq q\leq n$ (or, more precisely, by taking the closure of the locus where equality holds).  An example is shown in Figure~\ref{f.diagram-ex1}, where the entries determined by isotropicity are represented by an ``$\times$''; cf.~\cite[\S6.1]{fp}.

\begin{figure}[ht]
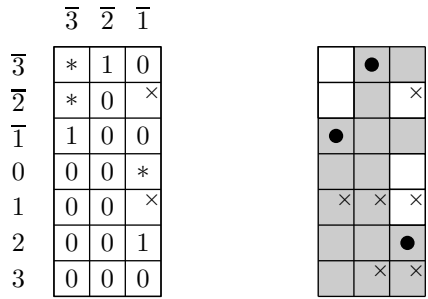

\pspicture(50,60)(-50,-50)

\psset{unit=1.35pt}

\pspolygon(-35,35)(-5,35)(-5,-35)(-35,-35)(-35,35)

\psline{-}(-5,35)(-35,35)
\psline{-}(-5,25)(-35,25)
\psline{-}(-5,15)(-35,15)
\psline{-}(-5,5)(-35,5)
\psline{-}(-5,-5)(-35,-5)
\psline{-}(-5,-15)(-35,-15)
\psline{-}(-5,-25)(-35,-25)
\psline{-}(-5,-35)(-35,-35)

\psline{-}(-5,35)(-5,-35)
\psline{-}(-15,35)(-15,-35)
\psline{-}(-25,35)(-25,-35)
\psline{-}(-35,35)(-35,-35)

\footnotesize{
\rput(-45,30){$\bar{3}$}
\rput(-45,20){$\bar{2}$}
\rput(-45,10){$\bar{1}$}
\rput(-45,0){$0$}
\rput(-45,-10){$1$}
\rput(-45,-20){$2$}
\rput(-45,-30){$3$}

\rput[b](-30,40){$\bar{3}$}
\rput[b](-20,40){$\bar{2}$}
\rput[b](-10,40){$\bar{1}$}
}

\rput(-30,30){$*$}
\rput(-30,20){$*$}
\rput(-30,10){\footnotesize{$1$}}
\rput(-30,0){\footnotesize{$0$}}
\rput(-30,-10){\footnotesize{$0$}}
\rput(-30,-20){\footnotesize{$0$}}
\rput(-30,-30){\footnotesize{$0$}}

\rput(-20,30){\footnotesize{$1$}}
\rput(-20,20){\footnotesize{$0$}}
\rput(-20,10){\footnotesize{$0$}}
\rput(-20,0){\footnotesize{$0$}}
\rput(-20,-10){\footnotesize{$0$}}
\rput(-20,-20){\footnotesize{$0$}}
\rput(-20,-30){\footnotesize{$0$}}

\rput(-10,30){\footnotesize{$0$}}
\rput[bl](-10,20){\tiny{$\times$}}
\rput(-10,10){\footnotesize{$0$}}
\rput(-10,0){$*$}
\rput[bl](-10,-10){\tiny{$\times$}}
\rput(-10,-20){\footnotesize{$1$}}
\rput(-10,-30){\footnotesize{$0$}}

\endpspicture
\pspicture(50,60)(-50,-50)

\psset{unit=1.35pt}

\pspolygon[fillstyle=solid,fillcolor=lightgray,linecolor=lightgray](-35,35)(-5,35)(-5,-35)(-35,-35)(-35,35)

\psline{-}(-5,35)(-35,35)
\psline{-}(-5,25)(-35,25)
\psline{-}(-5,15)(-35,15)
\psline{-}(-5,5)(-35,5)
\psline{-}(-5,-5)(-35,-5)
\psline{-}(-5,-15)(-35,-15)
\psline{-}(-5,-25)(-35,-25)
\psline{-}(-5,-35)(-35,-35)

\psline{-}(-5,35)(-5,-35)
\psline{-}(-15,35)(-15,-35)
\psline{-}(-25,35)(-25,-35)
\psline{-}(-35,35)(-35,-35)

\pscircle*(-30,10){2}
\pscircle*(-20,30){2}
\pscircle*(-10,-20){2}

\whitebox(-35,15)
\whitebox(-35,25)
\whitebox(-15,15)
\whitebox(-15,-5)
\whitebox(-15,-15)

\rput[bl](-30,-10){\tiny{$\times$}}
\rput[bl](-20,-10){\tiny{$\times$}}
\rput[bl](-10,-10){\tiny{$\times$}}
\rput[bl](-20,-30){\tiny{$\times$}}
\rput[bl](-10,-30){\tiny{$\times$}}
\rput[bl](-10,20){\tiny{$\times$}}

\endpspicture

\caption{Diagram of the signed permutation $\bar{2}\;3\;1$.\label{f.diagram-ex1}}
\end{figure}

The diagram of a signed permutation $w$ is a combinatorial abstraction of the matrix representing $E_\bullet$ (see Figures~\ref{f.diagram-ex1} and \ref{f.ex-intro}).  Its definition, given in detail in \S\ref{ss.essential-signed}, is similar to the diagram of an ordinary permutation, with the additional feature of markings ``$\times$''; the white boxes not containing an $\times$ form the {\it diagram}, while all the white boxes form the {\it extended diagram}.

In terms of the matrix, the rank $r_w(p,q)$ is equal to the number of dots in the region weakly southwest of the box $(q,\bar{p})$.  The essential set $\Ess(w)$ consists of certain {\it basic triples} $(k,p,q)$, where $k=r_w(p,q)$ and the box $(q-1,\bar{p})$ is a southeast corner of the (extended) diagram of $w$.  In contrast with ordinary permutations, however, not all southeast corners give essential conditions for a signed permutation.  There are two exceptions, stemming from two simple linear-algebraic facts about isotropic subspaces: first,
 the conditions $\dim(E_p\cap F_q) \geq k$ and $\dim(E_p^\perp \cap F_q^\perp) \geq k+p+q-1$ are equivalent; and second, for $k,p,q>0$ 
the condition $\dim(E_p\cap F_q)\geq k$ is implied by $\dim(E_p\cap F_q^\perp) \geq k+q-1$. 
An example is shown in Figure~\ref{f.ex-intro}.    
See Definition~\ref{d.essential} for details.

\begin{figure}[ht]
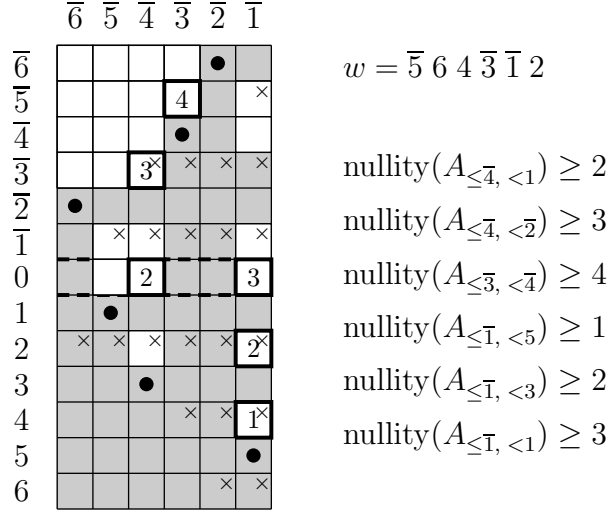

\pspicture(70,110)(-40,-90)

\psset{unit=1.35pt}

\pspolygon[fillstyle=solid,fillcolor=lightgray,linecolor=lightgray](-65,65)(-5,65)(-5,-65)(-65,-65)(-65,65)

\psline{-}(-5,65)(-65,65)
\psline{-}(-5,55)(-65,55)
\psline{-}(-5,45)(-65,45)
\psline{-}(-5,35)(-65,35)
\psline{-}(-5,25)(-65,25)
\psline{-}(-5,15)(-65,15)
\psline{-}(-5,5)(-65,5)
\psline{-}(-5,-5)(-65,-5)
\psline{-}(-5,-15)(-65,-15)
\psline{-}(-5,-25)(-65,-25)
\psline{-}(-5,-35)(-65,-35)
\psline{-}(-5,-45)(-65,-45)
\psline{-}(-5,-55)(-65,-55)
\psline{-}(-5,-65)(-65,-65)

\psline{-}(-5,65)(-5,-65)
\psline{-}(-15,65)(-15,-65)
\psline{-}(-25,65)(-25,-65)
\psline{-}(-35,65)(-35,-65)
\psline{-}(-45,65)(-45,-65)
\psline{-}(-55,65)(-55,-65)
\psline{-}(-65,65)(-65,-65)

\rput(-75,60){$\bar{6}$}
\rput(-75,50){$\bar{5}$}
\rput(-75,40){$\bar{4}$}
\rput(-75,30){$\bar{3}$}
\rput(-75,20){$\bar{2}$}
\rput(-75,10){$\bar{1}$}
\rput(-75,0){$0$}
\rput(-75,-10){$1$}
\rput(-75,-20){$2$}
\rput(-75,-30){$3$}
\rput(-75,-40){$4$}
\rput(-75,-50){$5$}
\rput(-75,-60){$6$}

\rput[b](-60,70){$\bar{6}$}
\rput[b](-50,70){$\bar{5}$}
\rput[b](-40,70){$\bar{4}$}
\rput[b](-30,70){$\bar{3}$}
\rput[b](-20,70){$\bar{2}$}
\rput[b](-10,70){$\bar{1}$}

\psline[linestyle=dashed,linewidth=1.5pt]{-}(-5,5)(-65,5)
\psline[linestyle=dashed,linewidth=1.5pt]{-}(-5,-5)(-65,-5)

\pscircle*(-10,-50){2}
\pscircle*(-20,60){2}
\pscircle*(-30,40){2}
\pscircle*(-40,-30){2}
\pscircle*(-50,-10){2}
\pscircle*(-60,20){2}

\whitebox(-65,55)
\whitebox(-65,45)
\whitebox(-65,35)
\whitebox(-65,25)

\whitebox(-55,55)
\whitebox(-55,45)
\whitebox(-55,35)
\whitebox(-55,25)
\whitebox(-55,5)
\whitebox(-55,-5)

\whitebox(-45,55)
\whitebox(-45,45)
\whitebox(-45,35)
\whitebox(-45,25)
\whitebox(-45,5)
\whitebox(-45,-5)
\whitebox(-45,-25)

\whitebox(-35,55)
\whitebox(-35,45)

\whitebox(-15,45)
\whitebox(-15,35)
\whitebox(-15,5)
\whitebox(-15,-5)
\whitebox(-15,-25)
\whitebox(-15,-45)

\rput[bl](-60,-20){\tiny{$\times$}}
\rput[bl](-50,-20){\tiny{$\times$}}
\rput[bl](-40,-20){\tiny{$\times$}}
\rput[bl](-30,-20){\tiny{$\times$}}
\rput[bl](-20,-20){\tiny{$\times$}}
\rput[bl](-10,-20){\tiny{$\times$}}

\rput[bl](-50,10){\tiny{$\times$}}
\rput[bl](-40,10){\tiny{$\times$}}
\rput[bl](-30,10){\tiny{$\times$}}
\rput[bl](-20,10){\tiny{$\times$}}
\rput[bl](-10,10){\tiny{$\times$}}

\rput[bl](-40,30){\tiny{$\times$}}
\rput[bl](-30,30){\tiny{$\times$}}
\rput[bl](-20,30){\tiny{$\times$}}
\rput[bl](-10,30){\tiny{$\times$}}

\rput[bl](-30,-40){\tiny{$\times$}}
\rput[bl](-20,-40){\tiny{$\times$}}
\rput[bl](-10,-40){\tiny{$\times$}}

\rput[bl](-20,-60){\tiny{$\times$}}
\rput[bl](-10,-60){\tiny{$\times$}}

\rput[bl](-10,50){\tiny{$\times$}}

\psline[linewidth=1.5pt]{-}(-45,-5)(-45,5)(-35,5)(-35,-5)(-45,-5)
\rput(-40,0){\footnotesize{$2$}}

\psline[linewidth=1.5pt]{-}(-45,25)(-45,35)(-35,35)(-35,25)(-45,25)
\rput(-40,30){\footnotesize{$3$}}

\psline[linewidth=1.5pt]{-}(-35,45)(-35,55)(-25,55)(-25,45)(-35,45)
\rput(-30,50){\footnotesize{$4$}}

\psline[linewidth=1.5pt]{-}(-15,-45)(-15,-35)(-5,-35)(-5,-45)(-15,-45)
\rput(-10,-40){\footnotesize{$1$}}

\psline[linewidth=1.5pt]{-}(-15,-25)(-15,-15)(-5,-15)(-5,-25)(-15,-25)
\rput(-10,-20){\footnotesize{$2$}}

\psline[linewidth=1.5pt]{-}(-15,-5)(-15,5)(-5,5)(-5,-5)(-15,-5)
\rput(-10,0){\footnotesize{$3$}}

\rput[l](15,60){$w=\bar{5}\;6\;4\;\bar{3}\;\bar{1}\;2$}
\rput[l](15,30){$\Null(A_{\leq\bar{4},\;<1}) \geq 2$}
\rput[l](15,15){$\Null(A_{\leq\bar{4},\;<\bar{2}}) \geq 3$}
\rput[l](15,0){$\Null(A_{\leq\bar{3},\;<\bar{4}}) \geq 4$}
\rput[l](15,-15){$\Null(A_{\leq\bar{1},\;<5}) \geq 1$}
\rput[l](15,-30){$\Null(A_{\leq\bar{1},\;<3}) \geq 2$}
\rput[l](15,-45){$\Null(A_{\leq\bar{1},\;<1}) \geq 3$}

\endpspicture

\caption{Diagram and essential rank conditions.  In the diagram, each essential position is labelled with the number of dots strictly south and weakly west of it.  \label{f.ex-intro}}
\end{figure}

The key geometric fact about essential sets is the following:
\begin{thmA}[Corollary~\ref{c.essential-b}]
The degeneracy locus corresponding to a signed permutation $w$ is defined by the conditions $\dim(E_p \cap F_q) \geq k$, as $(k,p,q)$ ranges over $\Ess(w)$.
\end{thmA}

Taking $X$ to be the orthogonal flag variety, $E_\bullet$ the tautological isotropic flag, and $F_\bullet$ a flag of trivial isotropic bundles, Theorem~A says the Schubert variety corresponding to $w$ is determined by rank conditions from the essential set.  This particular case suffices to prove the general degeneracy locus statement.

The theorem can be interpreted in terms of rank conditions on matrices, as follows.  Take a standard basis $e_n,\ldots,e_1, e_0, e_{\bar{1}}, \ldots, e_{\bar{n}}$ for a vector space of dimension $2n+1$, and fix the symmetric bilinear form defined by $\langle e_i, e_{\bar\jmath} \rangle = \delta_{i,j}$.  Let $X$ be the set of $(2n+1)\times n$ full-rank matrices with isotropic and pairwise orthogonal columns, with rows indexed $\bar{n},\ldots,0,\ldots,n$ (top to bottom) and columns indexed $\bar{n},\ldots,\bar{1}$ (left to right).  

\begin{cor*}
Given a signed permutation $w$, with rank function $r_w(p,q)$ as defined above, let $\Omega_w\subseteq X$ be the subset
\[
  \Omega_w = \{ A \in X \,|\, \Null (A_{\leq \bar{p},\; < q}) \geq r_w(p,q) \text{ for all } 1\leq p\leq n, \bar{n} \leq q\leq n \},
\]
where $A_{\leq \bar{p},\; < q}$ is the (upper-left) submatrix on columns $\bar{n},\ldots,\bar{p}$ and rows $\bar{n},\ldots,q-1$.  Then the conditions
\[
  \Null (A_{\leq \bar{p},\; < q}) \geq k  \; \text{ for }\; (k,p,q) \in \Ess(w)
\]
suffice to determine $\Omega_w$, and they form a minimal set with this property.
\end{cor*}

\noindent
To deduce this from Theorem~A, observe that the the flag variety $SO_{2n+1}/B$ can be represented as a quotient of the set of matrices $X$, and the composed map $E_p \hookrightarrow V \twoheadrightarrow V/F_q$ is represented by the submatrix $A_{\leq\bar{p},\; <q}$.  The set $\Omega_w$ is the preimage of the corresponding Schubert variety under the quotient map.  An example is shown in Figure~\ref{f.ex-intro}.

With Theorem~A in mind, we define the {\it basic signed permutation} $w(k,p,q)$ associated to a basic triple $(k,p,q)$ so that the Schubert variety $\Omega_{w(k,p,q)}$ is defined by a single rank condition (see \S\ref{ss.basic} for the combinatorial definition).  The geometric statement of Theorem~A is a direct consequence of a combinatorial statement about Bruhat order---a signed permutation $w$ is the supremum of the basic signed permutations corresponding to the elements of $\Ess(w)$ (Theorem~\ref{t.essential-b}).  This, in turn, relies on our main result about essential sets:
\begin{thmB}[Theorem~\ref{t.ess-diagram}]
The essential set of a signed permutation $w$ corresponds to the set of basic signed permutations which are maximal among all those below $w$ in Bruhat order.
\end{thmB}

Theorem~B also suggests a link with a general poset-theoretic definition of essential set introduced by Reiner, Woo, and Yong in the course of finding descriptions for the cohomology rings of Schubert varieties \cite{rwy}.  We establish a bijective correspondence in Proposition~\ref{p.rwy-equivalent}.  Along the way, we also show that the set of basic signed permutations is in bijection with the {\em base} of the group of signed permutations, which was studied by Lascoux and Sch\"utzenberger \cite{ls2} and Geck and Kim \cite{gk} in the context of characterizing Bruhat order of Weyl groups.

From an alternative point of view, our diagrammatic definition together with Theorem~B may be taken as an efficient method for computing the essential set as defined in \cite{rwy}, and hence finding generators for the ideal of the cohomology ring of a Schubert variety.  From the perspective of \cite{ls2} and \cite{gk}, this also leads to an efficient means of comparing elements in Bruhat order on signed permutations.

It would be interesting to see signed analogues of other properties of diagrams for (ordinary) permutations.  For instance, ``balanced fillings'' of the diagram of a permutation are in bijection with reduced words, and this leads to a formula for the corresponding Stanley symmetric function \cite{klr,fgrs}.  Is there a similar story for diagrams of signed permutations?  (In \cite{hamaker}, a different notion of diagram is proposed, as well as a set of balanced fillings in bijection with reduced words; the tradeoff is that one must label all the boxes of the diagram of a permutation in $S_{2n}$, rather than the smaller set of boxes we consider here.)  In a different direction, Eriksson and Linusson characterized the sets that can arise as essential sets of permutations, and used this to give an efficient algorithm for reconstructing a permutation from its essential set; in fact, only a subset of the essential set is required \cite{el}.  A result of this kind for signed permutations would be useful, as well.

Much of our discussion of diagrams and essential sets for signed permutations has a ``type B'' flavor, but this is primarily for notational convenience.  The group $W_n$ of signed permutations is also the Weyl group of type C, and indeed, one can recast all the results in a form more adapted to the symplectic group; an indication is given at the end of the article.  Type D, however, is another matter---the Weyl group is a subgroup of $W_n$, and it does not satisfy the ``dissective'' condition used in \S\ref{ss.bases}.  It would be interesting to work out the analogous story for type D, but this case seems sufficiently distinct to be treated separately.

\medskip
\noindent
{\it Acknowledgements.}  I am grateful to Sara Billey, Zach Hamaker, and Alexander Woo for helpful comments and suggestions.

The initial impetus for this paper came from the study of degeneracy loci and vexillary signed permutations in an ongoing joint project with William Fulton.  I thank him for detailed comments on the manuscript, as well as for an enjoyable and fruitful collaboration.

\medskip
\noindent
{\it Notation.}  
Our conventions and notation for signed permutations are very close to those described in \cite[\S8.1]{bb}.  We review them here.

We will consider permutations of (positive and negative) integers
\[
  \ldots, \bar{n}, \ldots, \bar{2},\bar{1},0,1,2,\ldots,n,\ldots,
\]
using the bar to denote a negative sign, and we take the natural order on them, as above.  All permutations are finite, in the sense that $v(m)=m$ whenever $|m|$ is sufficiently large.  We generally write permutations in $S_{2n+1}$ using one-line notation, by listing the values $v(\bar{n})\;v(\bar{n-1}) \cdots v(n)$.

A {\em signed permutation} is a permutation $w$ of these symbols with the property that for each $i$, $w(\bar\imath) = \bar{w(i)}$.  A signed permutation is in $W_n$ if $w(m) = m$ for all $m>n$; this is a group isomorphic to the hyperoctahedral group, the Weyl group of types $B_n$ and $C_n$.  When writing signed permutations in one-line notation, we only list the values on positive integers: $w\in W_n$ is represented as $w(1)\;w(2)\;\cdots\;w(n)$.  For example, $w=\bar{2}\;1\;\bar{3}$ is a signed permutation in $W_3$, and $w(\bar{3}) = 3$ since $w(3)=\bar{3}$.  The \emph{length} of a signed permutation can be computed as
\[
  \ell(w) = \#\{ 0<i<j \,|\, w(i)>w(j) \} + \#\{ 0<i\leq j \,|\, w(i)+w(j)<0 \}.
\]
The \emph{longest element} in $W_n$, denoted $w_\circ^{(n)}$, is $\bar{1}\;\bar{2}\;\cdots\;\bar{n}$, and has length $n^2$.

The definition of $W_n$ presents it as embedded in the symmetric group $S_{2n+1}$, considering the latter as the group of all permutations of the integers $\bar{n},\ldots,0,\ldots,n$.  We will write $\iota\colon W_n \hookrightarrow S_{2n+1}$ for emphasis when a signed permutation is considered as a full permutation.  Specifically, $\iota$ sends $w=w(1)\;w(2)\;\cdots\;w(n)$ to the permutation
\[
  \bar{w(n)}\;\cdots\;\bar{w(2)}\;\bar{w(1)}\;0\;w(1)\;w(2)\;\cdots\;w(n)
\]
in $S_{2n+1}$.  (There is a similar embedding $\iota'\colon W_n \hookrightarrow S_{2n}$, defined by omitting the value ``$w(0)=0$''.)

Using the natural inclusions $W_n \subset W_{n+1} \subset \cdots$, we have the infinite Weyl group $W_\infty = \bigcup W_n$, and we will sometimes refer to elements $w\in W_\infty$, when $n$ is understood or irrelevant.  The embeddings are compatible with the corresponding inclusions $S_{2n+1} \subset S_{2n+3} \subset \cdots$.

A permutation $v$ has a \define{descent} at position $i$ if $v(i)>v(i+1)$; here $i$ may be any integer.  The same definition applies to signed permutations $w$, but we only consider descents at positions $i\geq 0$, following the convention of recording the values of $w$ only on positive integers.  A descent at $0$ simply means that $w(1)$ is negative.  For example, $w = \bar{2}\;1\;\bar{3}$ has descents at $0$ and $2$, while $\iota(w) = 3\;\bar{1}\;2\;0\;\bar{2}\;1\;\bar{3}$ has descents at $-3$, $-1$, $0$, and $2$.

\section{Defining the essential set}\label{s.diagrams}

Our definitions of diagram and essential set for signed permutations are modeled on the analogous notions for permutations, so we briefly review that case first.  The conventions we use here are adapted to make the transition to signed permutations simpler, so they differ slightly from those used in \cite{fulton} and elsewhere.

\subsection{Essential sets for $S_{2n+1}$}\label{ss.essential-typeA}
We will consider arrays of boxes with rows and columns indexed by integers $\{\bar{n},\ldots,\bar{1},0,1,\ldots,n\}$.  The \emph{permutation matrix} associated to $v\in S_{2n+1}$ has dots in positions $(v(i),i)$, for $\bar{n}\leq i\leq n$, and is empty otherwise.  The \emph{diagram} of $v$ is the collection of boxes that remain after striking out those which are (weakly) south or east of a dot in the permutation matrix.  The number of boxes in the diagram is equal to the length of the permutation.  

The \emph{rank function} of a permutation is defined by
\begin{equation}\label{e.rank-typeA}
  r_v(p,q) = \#\{ i\leq \bar{p} \,|\, v(i) \geq q \},
\end{equation}
for $\bar{n} \leq p,q\leq n$.  This is also the number of dots strictly south and weakly west of the box $(q-1,\bar{p})$ in the permutation matrix of $v$.

Adapting terminology from \cite{fulton} and \cite{rwy}, the \define{essential positions} of $v$ are the pairs $(p,q)$ such that the box $(q-1,\bar{p})$ is a southeast (SE) corner of the diagram of $v$.  (This apparently awkward notation is easier to understand in terms of the geometry explained in the next paragraph, together with the diagrams.)   The \define{essential set} of $v$ is the set $\Ess(v)$ of $(k,p,q)$ such that $(p,q)$ is an essential position, and $k=r_v(p,q)$; it is obviously in bijection with the set of essential positions of $v$, but it will be useful to preserve the extra information of the rank function.  

It is often helpful to think of flags of subspaces $E_\bullet$ and $F_\bullet$ in a $(2n+1)$-dimensional vector space $V$, with basis $e_{\bar{n}},\ldots,e_{\bar{1}},0,e_1,\ldots,e_n$.  The $n\times n$ array corresponds to a matrix with respect to this basis.  The space $E_p$ is spanned by the columns labelled $\bar{p},\bar{p+1},\ldots$---i.e., from $\bar{p}$ to the left---and the space $F_q$ is spanned by the basis vectors $e_q,e_{q+1},\ldots$---i.e., from $q$ down.  The number $k=r_v(p,q)$ records the dimension of $E_p \cap F_q$; it is also the corank of the map $E_p \to V/F_q$, and the box $(q-1,\bar{p})$ is the SE corner of submatrix corresponding to this map.

Figure~\ref{f.diagram1} illustrates these notions, for $v = \bar{1}\;\bar{3}\;2\;0\;\bar{2}\;3\;1$, with the numbers $r_v(p,q)$ placed in the boxes $(q-1,\bar{p})$ corresponding to the essential positions.  The essential set is
\[
\Ess(v) = \{\,(1,3,\bar{1}),\,(1,1,2),\,(3,0,\bar{1}),\,(2,\bar{2},2)\,\}.
\]

\begin{figure}[ht]
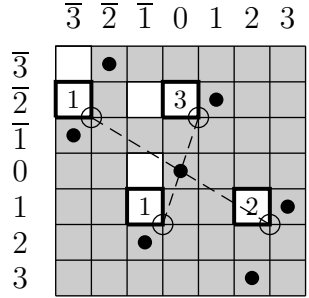


\pspicture(80,80)(-70,-50)

\psset{unit=1.35pt}

\pspolygon[fillstyle=solid,fillcolor=lightgray,linecolor=lightgray](-35,35)(35,35)(35,-35)(-35,-35)(-35,35)

\psline{-}(35,35)(-35,35)
\psline{-}(35,25)(-35,25)
\psline{-}(35,15)(-35,15)
\psline{-}(35,5)(-35,5)
\psline{-}(35,-5)(-35,-5)
\psline{-}(35,-15)(-35,-15)
\psline{-}(35,-25)(-35,-25)
\psline{-}(35,-35)(-35,-35)

\psline{-}(35,35)(35,-35)
\psline{-}(25,35)(25,-35)
\psline{-}(15,35)(15,-35)
\psline{-}(5,35)(5,-35)
\psline{-}(-5,35)(-5,-35)
\psline{-}(-15,35)(-15,-35)
\psline{-}(-25,35)(-25,-35)
\psline{-}(-35,35)(-35,-35)

\pscircle*(-30,10){2}
\pscircle*(-20,30){2}
\pscircle*(-10,-20){2}
\pscircle*(0,0){2}
\pscircle*(10,20){2}
\pscircle*(20,-30){2}
\pscircle*(30,-10){2}

\whitebox(-35,15)
\whitebox(-35,25)
\whitebox(-15,15)
\whitebox(-5,15)
\whitebox(-15,-5)
\whitebox(-15,-15)
\whitebox(15,-15)

\psline[linewidth=1.5pt]{-}(-35,15)(-35,25)(-25,25)(-25,15)(-35,15)

\psline[linewidth=1.5pt]{-}(-15,-15)(-15,-5)(-5,-5)(-5,-15)(-15,-15)

\psline[linewidth=1.5pt]{-}(-5,15)(-5,25)(5,25)(5,15)(-5,15)

\psline[linewidth=1.5pt]{-}(15,-15)(15,-5)(25,-5)(25,-15)(15,-15)

\pscircle(-25,15){3}
\pscircle(25,-15){3}

\pscircle(-5,-15){3}
\pscircle(5,15){3}

\psline[linestyle=dashed]{-}(-25,15)(25,-15)
\psline[linestyle=dashed]{-}(-5,-15)(5,15)

\rput(-45,30){$\bar{3}$}
\rput(-45,20){$\bar{2}$}
\rput(-45,10){$\bar{1}$}
\rput(-45,0){$0$}
\rput(-45,-10){$1$}
\rput(-45,-20){$2$}
\rput(-45,-30){$3$}

\rput[b](-30,40){$\bar{3}$}
\rput[b](-20,40){$\bar{2}$}
\rput[b](-10,40){$\bar{1}$}
\rput[b](0,40){$0$}
\rput[b](10,40){$1$}
\rput[b](20,40){$2$}
\rput[b](30,40){$3$}

\rput(-10,-10){\footnotesize{$1$}}
\rput(-30,20){\footnotesize{$1$}}

\rput(20,-10){\footnotesize{$2$}}
\rput(0,20){\footnotesize{$3$}}

\endpspicture

\caption{Diagram for $v = \bar{1}\;\bar{3}\;2\;0\;\bar{2}\;3\;1$, with essential set highlighted.  The circled corners connected with dashed lines illustrate the symmetry of Lemma~\ref{l.symmetric-essential}. \label{f.diagram1}}
\end{figure}

An equivalent, numerical description of essential positions is useful: $(a,b)$ is a SE corner of the diagram of $v$ if and only if
\begin{align}
  v^{-1}(a) > b \; \text{ and } \; v(b) > a;  \quad  v^{-1}(a+1) \leq b \;\text{ and } \; v(b+1) \leq a . \label{e.corners}
\end{align}
This is also equivalent to requiring the conditions
\begin{align}
 v(b) &> a \geq v(b+1)  \quad \text{ and } \label{e.descents1} \\ 
 v^{-1}(a) &>b \geq v^{-1}(a+1) \tag{\theequation$'$}. \label{e.descents2}
\end{align}
In other words, $v$ has a descent at $b$, with $a$ lying in the interval of the jump, and $v^{-1}$ has a descent at $a$, with $b$ lying in the interval of the jump.  The essential positions are recovered as those $(p,q)$ such that $(q-1,\bar{p})$ satisfies \eqref{e.corners}, and the essential set is recovered by including $k=r_v(p,q)$.  

A \define{basic triple (of type A)} is a triple of integers $(k,p,q)$ such that $k>\max\{0,1-p-q\}$.    
It is easy to see that the elements of $\Ess(v)$ are basic triples.  

\subsection{Essential sets for $W_n$}\label{ss.essential-signed}

By convention, a signed permutation $w$ is written in terms of its values on positive integers, but since $w(\bar{\imath}) = \bar{w(i)}$, it is also determined by its values on negative integers.  Given $w\in W_n$, its corresponding matrix is the $(2n+1)\times n$ array of boxes, with rows indexed by $\{\bar{n},\ldots,n\}$ and columns indexed by $\{\bar{n},\ldots,\bar{1}\}$, with dots in the boxes $(w(i),i)$ for $\bar{n}\leq i\leq \bar{1}$.  
For each dot, we also place an ``$\times$'' in the same column and opposite row, as well as in the boxes to the right of this $\times$.  (To be precise, an $\times$ is placed in those boxes $(a,b)$ such that $a=\bar{w(i)}$ for some $i\leq b$.  This is similar to the matrix representation of Schubert cells described in \cite[\S6.1]{fp}.)

The \define{extended diagram} $D^+(w)$ of a signed permutation $w$ is the collection of boxes in the $(2n+1)\times n$ rectangle that remain after striking out those which are south or east of a dot.  The \define{diagram} $D(w) \subseteq D^+(w)$ is the subset of boxes of the extended diagram that are not marked with an $\times$.  As for permutations, the number of boxes of $D(w)$ is equal to the length of $w$.  In fact, the inversions of $w$ are in bijection with the boxes of $D(w)$: writing the positive simple roots for $\mathrm{B}_n$ as $\ee_1, \ee_2-\ee_1,\ldots,\ee_n-\ee_{n-1}$, and using the convention $\ee_{\bar\imath}=-\ee_i$, the box in position $(i,j)$ lies in $D(w)$ if and only if $\ee_i-\ee_{w(j)}$ is an inversion.

The matrix and extended diagram come from the corresponding notions for ordinary permutations, via the embeddeding $\iota\colon W_n\hookrightarrow S_{2n+1}$: the matrix and extended diagram of $w\in W_n$ correspond to the first $n$ columns of the matrix and diagram for $\iota(w)$.  See Figure~\ref{f.diagram2}, where the unshaded boxes are those of $D^+(w)$.

The rank function of a signed permutation is defined by
\begin{equation}
  r_w(p,q) = \#\{i\geq p\,|\, w(i) \leq \bar{q}\},
\end{equation}
for $1\leq p\leq n$ and $\bar{n}\leq q\leq n$.  Using the symmetry $w(\bar\imath) = \bar{w(i)}$, this is also equal to $\#\{ i\leq \bar{p}\,|\, w(i) \geq q\}$, so the rank functions $r_w$ and $r_{\iota(w)}$ agree where both are defined.

Integers $(k,p,q)$ form a \define{basic triple (of type B)} if the same inequality as for type A is satisfied, i.e., $k>\max\{0,1-p-q\}$, together with three additional conditions: $p>0$, $q\neq 0$, and if $p=1$ then $q>0$.  

Before discussing essential sets for signed permutations, we record a simple observation.  When $v=\iota(w)$ lies in the image of $\iota\colon W_n \hookrightarrow S_{2n+1}$, its essential set is symmetric: the lower-right corners of the boxes corresponding to essential positions are reflected about the center.  (The example of $w=\bar{2}\;3\;1$ is shown in Figure~\ref{f.diagram1}, where these corners are circled.)  To describe this symmetry more precisely, given a basic triple $(k,p,q)$ of type A, let $(k,p,q)^\perp = (k+p+q-1,\bar{p}+1,\bar{q}+1)$ be its \define{reflection}; this defines an involution on basic triples.

\begin{lemma}\label{l.symmetric-essential}
For $w\in W_n$, the essential set of $\iota(w)\in S_{2n+1}$ is preserved by reflection.  That is, $(k,p,q)$ is in $\Ess(\iota(w))$ if and only if $(k,p,q)^\perp$ is in $\Ess(\iota(w))$.
\end{lemma}

\begin{proof} 
First, we show that the SE corners of the diagram are centrally symmetric.  If $v=\iota(w)$ and $(a,b)$ is a SE corner of the diagram of $v$, then using $v(\bar{k}) = \bar{v(k)}$, the inequalities \eqref{e.corners} become
\begin{align*}
  \bar{b} > v^{-1}(\bar{a}) \quad &\text{ and } \quad \bar{a} > v(\bar{b}); \\
  \bar{b} \leq v^{-1}(\bar{a+1}) \quad &\text{ and } \quad \bar{a} \leq v(\bar{b+1}).
\end{align*}
Equivalently,
\begin{align*}
  v^{-1}(\bar{a})\leq \bar{b+1} \quad &\text{ and } \quad v(\bar{b}) \leq \bar{a+1}; \\
  v^{-1}(\bar{a+1}) > \bar{b+1} \quad &\text{ and } \quad v(\bar{b+1}) > \bar{a+1}.
\end{align*}
These are the conditions for $(\bar{a+1},\bar{b+1})$ to be a SE corner corresponding to an essential position.

This shows that $(k,p,q)$ is in $\Ess(v)$ if and only if $(k',\bar{p}+1,\bar{q}+1)$ is in $\Ess(v)$, where $k' = r_v(\bar{p}+1,\bar{q}+1)$.  It remains to check that this rank is equal to $k+p+q-1$, as claimed.  This is easily done by examining a permutation matrix and using $r_v(p,q) = \#\{ i\leq \bar{p}\,|\, w(i) \geq q \} = \#\{ i\geq p\,|\, w(i) \leq \bar{q} \}$.
\end{proof}

\begin{figure}
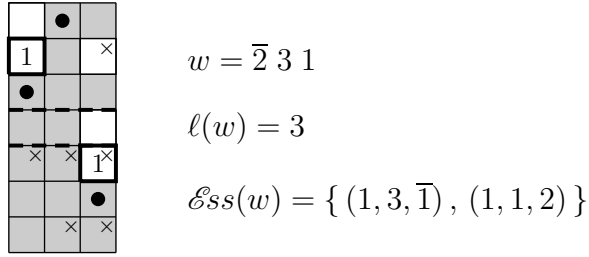

\pspicture(100,50)(-50,-50)

\psset{unit=1.35pt}

\pspolygon[fillstyle=solid,fillcolor=lightgray,linecolor=lightgray](-35,35)(-5,35)(-5,-35)(-35,-35)(-35,35)

\psline{-}(-5,35)(-35,35)
\psline{-}(-5,25)(-35,25)
\psline{-}(-5,15)(-35,15)
\psline{-}(-5,5)(-35,5)
\psline{-}(-5,-5)(-35,-5)
\psline{-}(-5,-15)(-35,-15)
\psline{-}(-5,-25)(-35,-25)
\psline{-}(-5,-35)(-35,-35)

\psline{-}(-5,35)(-5,-35)
\psline{-}(-15,35)(-15,-35)
\psline{-}(-25,35)(-25,-35)
\psline{-}(-35,35)(-35,-35)

\pscircle*(-30,10){2}
\pscircle*(-20,30){2}
\pscircle*(-10,-20){2}

\whitebox(-35,15)
\whitebox(-35,25)
\whitebox(-15,15)
\whitebox(-15,-5)
\whitebox(-15,-15)

\psline[linestyle=dashed,linewidth=1.5pt]{-}(-5,5)(-35,5)
\psline[linestyle=dashed,linewidth=1.5pt]{-}(-5,-5)(-35,-5)

\psline[linewidth=1.5pt]{-}(-35,15)(-35,25)(-25,25)(-25,15)(-35,15)

\psline[linewidth=1.5pt]{-}(-15,-15)(-15,-5)(-5,-5)(-5,-15)(-15,-15)

\rput[bl](-30,-10){\tiny{$\times$}}
\rput[bl](-20,-10){\tiny{$\times$}}
\rput[bl](-10,-10){\tiny{$\times$}}
\rput[bl](-20,-30){\tiny{$\times$}}
\rput[bl](-10,-30){\tiny{$\times$}}
\rput[bl](-10,20){\tiny{$\times$}}

\rput(-10,-10){\footnotesize{$1$}}
\rput(-30,20){\footnotesize{$1$}}

\rput[l](15,20){$w=\bar{2}\;3\;1$}
\rput[l](15,0){$\ell(w)=3$}
\rput[l](15,-20){$\Ess(w) = \{\, (1,3,\bar{1})\,,\,(1,1,2) \,\}$}

\endpspicture
\caption{Matrix, diagram, and essential set for a signed permutation.  (As a visual aid, the center row is indicated with dashed lines.) \label{f.diagram2}}
\end{figure}

The essential set of a signed permutation $w$ is a subset of the essential set of the corresponding permutation $\iota(w)$.

\begin{definition}\label{d.essential}
Let $w$ be a signed permutation.  The \define{essential set} of $w$ is the set $\Ess(w)$ of triples of integers $(k,p,q)$, with $k=r_w(p,q)$, such that $(q-1,\bar{p})$ is a SE corner of the extended diagram $D^+(w)$, with two exceptions: a SE corner does not contribute to $\Ess(w)$ if
\begin{enumerate}
\item it is in the rightmost column and (strictly) above the center row; or

\smallskip

\item it is not in the rightmost column and there is another SE corner above it, in the same column and ``opposite'' row---that is, in box $(\bar{q},\bar{p})$---and the two rank conditions differ by exactly $q-1$.
\end{enumerate}
The \define{essential positions} of $w$ are those pairs $(p,q)$ such that $(k,p,q) \in \Ess(w)$, for $k=r_w(p,q)$.  (As before, $\Ess(w)$ is clearly in bijection with the set of essential positions.)
\end{definition}

The two exceptions can be phrased more formally as follows:
\begin{enumerate}
\item if $p=1$ and $q<0$, then $(k,p,q)$ is not in $\Ess(w)$;

\smallskip

\item if $p>1$ and $q>0$, and both $(q-1,\bar{p})$ and $(\bar{q},\bar{p})$ are SE corners, with $k=r_w(p,q) = r_w(p,\bar{q}+1)-q+1$, then $(k,p,q)$ is not in $\Ess(w)$.
\end{enumerate}

\noindent
The first exception is easily understood:  $w(0)=0$ prevents the diagram of $\iota(w)$ from having SE corners in this region, so such $(k,p,q)$ are not included in $\Ess(\iota(w))$.  (Compare Figures~\ref{f.diagram1} and \ref{f.diagram2}.)  Noting this, it is easy to see that $\Ess(w)\subseteq \Ess(\iota(w))$, and consequently, each $(k,p,q) \in \Ess(w)$ is a basic triple of type B.  (In fact, this also explains why those $(k,p,q)$ with $p=1$ and $q<0$ are excluded in the definition of basic triples of type B.)

The second exception has a geometric explanation.  For any $k,p,q>0$, the condition $\dim(E_p\cap F_q) \geq k$ is implied by the condition $\dim(E_p\cap F_q^\perp)\geq k+q-1$, because an isotropic subspace of $F_q^\perp/F_q$ must have dimension no greater than $q-1$.  The former condition is therefore redundant, and should be excluded from the essential set.

Examples illustrating Exception (ii) are shown in Figure~\ref{f.essential-ex1}.  In Figure~\ref{f.essential-ex1}(a), placing three dots strictly south and weakly west (SSW) of the box $(\bar{3},\bar{3})$ forces at least one of them to be SSW of the box $(2,\bar{3})$, so $(1,3,3)$ is not in the essential set.  In Figure~\ref{f.essential-ex1}(b), by contrast, both SE corners are required; the essential set is $\{(3,3,\bar{2}),\; (2,3,3)\}$.
\begin{figure}[t]
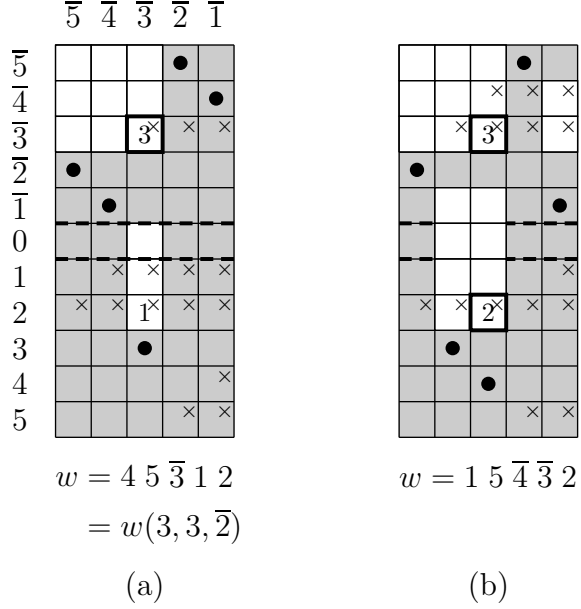

\pspicture(80,110)(-100,-140)

\psset{unit=1.35pt}

\pspolygon[fillstyle=solid,fillcolor=lightgray,linecolor=lightgray](-55,55)(-5,55)(-5,-55)(-55,-55)(-55,55)

\psline{-}(-5,55)(-55,55)
\psline{-}(-5,45)(-55,45)
\psline{-}(-5,35)(-55,35)
\psline{-}(-5,25)(-55,25)
\psline{-}(-5,15)(-55,15)
\psline{-}(-5,5)(-55,5)
\psline{-}(-5,-5)(-55,-5)
\psline{-}(-5,-15)(-55,-15)
\psline{-}(-5,-25)(-55,-25)
\psline{-}(-5,-35)(-55,-35)
\psline{-}(-5,-45)(-55,-45)
\psline{-}(-5,-55)(-55,-55)

\psline{-}(-5,55)(-5,-55)
\psline{-}(-15,55)(-15,-55)
\psline{-}(-25,55)(-25,-55)
\psline{-}(-35,55)(-35,-55)
\psline{-}(-45,55)(-45,-55)
\psline{-}(-55,55)(-55,-55)

\rput(-65,50){$\bar{5}$}
\rput(-65,40){$\bar{4}$}
\rput(-65,30){$\bar{3}$}
\rput(-65,20){$\bar{2}$}
\rput(-65,10){$\bar{1}$}
\rput(-65,0){$0$}
\rput(-65,-10){$1$}
\rput(-65,-20){$2$}
\rput(-65,-30){$3$}
\rput(-65,-40){$4$}
\rput(-65,-50){$5$}

\rput[b](-50,60){$\bar{5}$}
\rput[b](-40,60){$\bar{4}$}
\rput[b](-30,60){$\bar{3}$}
\rput[b](-20,60){$\bar{2}$}
\rput[b](-10,60){$\bar{1}$}

\pscircle*(-50,20){2}
\pscircle*(-40,10){2}
\pscircle*(-30,-30){2}
\pscircle*(-20,50){2}
\pscircle*(-10,40){2}

\whitebox(-55,45)
\whitebox(-55,35)
\whitebox(-55,25)

\whitebox(-45,45)
\whitebox(-45,35)
\whitebox(-45,25)

\whitebox(-35,45)
\whitebox(-35,35)
\whitebox(-35,25)
\whitebox(-35,-5)
\whitebox(-35,-15)
\whitebox(-35,-25)

\psline[linestyle=dashed,linewidth=1.5pt]{-}(-5,5)(-55,5)
\psline[linestyle=dashed,linewidth=1.5pt]{-}(-5,-5)(-55,-5)

\psline[linewidth=1.5pt]{-}(-35,25)(-35,35)(-25,35)(-25,25)(-35,25)
\rput(-30,30){\footnotesize{$3$}}

\rput(-30,-20){\footnotesize{$1$}}

\rput[bl](-50,-20){\tiny{$\times$}}
\rput[bl](-40,-20){\tiny{$\times$}}
\rput[bl](-30,-20){\tiny{$\times$}}
\rput[bl](-20,-20){\tiny{$\times$}}
\rput[bl](-10,-20){\tiny{$\times$}}
\rput[bl](-40,-10){\tiny{$\times$}}
\rput[bl](-30,-10){\tiny{$\times$}}
\rput[bl](-20,-10){\tiny{$\times$}}
\rput[bl](-10,-10){\tiny{$\times$}}
\rput[bl](-30,30){\tiny{$\times$}}
\rput[bl](-20,30){\tiny{$\times$}}
\rput[bl](-10,30){\tiny{$\times$}}
\rput[bl](-20,-50){\tiny{$\times$}}
\rput[bl](-10,-50){\tiny{$\times$}}
\rput[bl](-10,-40){\tiny{$\times$}}

\rput[l](-55,-65){$w = 4\;5\;\bar{3}\;1\;2$}
\rput[l](-46,-80){$ = w(3,3,\bar{2})$}

\rput(-30,-97){(a)}

\endpspicture
\pspicture(70,110)(-50,-140)

\psset{unit=1.35pt}

\pspolygon[fillstyle=solid,fillcolor=lightgray,linecolor=lightgray](-55,55)(-5,55)(-5,-55)(-55,-55)(-55,55)

\psline{-}(-5,55)(-55,55)
\psline{-}(-5,45)(-55,45)
\psline{-}(-5,35)(-55,35)
\psline{-}(-5,25)(-55,25)
\psline{-}(-5,15)(-55,15)
\psline{-}(-5,5)(-55,5)
\psline{-}(-5,-5)(-55,-5)
\psline{-}(-5,-15)(-55,-15)
\psline{-}(-5,-25)(-55,-25)
\psline{-}(-5,-35)(-55,-35)
\psline{-}(-5,-45)(-55,-45)
\psline{-}(-5,-55)(-55,-55)

\psline{-}(-5,55)(-5,-55)
\psline{-}(-15,55)(-15,-55)
\psline{-}(-25,55)(-25,-55)
\psline{-}(-35,55)(-35,-55)
\psline{-}(-45,55)(-45,-55)
\psline{-}(-55,55)(-55,-55)

\psline[linestyle=dashed,linewidth=1.5pt]{-}(-5,5)(-55,5)
\psline[linestyle=dashed,linewidth=1.5pt]{-}(-5,-5)(-55,-5)

\pscircle*(-50,20){2}
\pscircle*(-40,-30){2}
\pscircle*(-30,-40){2}
\pscircle*(-20,50){2}
\pscircle*(-10,10){2}

\whitebox(-55,45)
\whitebox(-55,35)
\whitebox(-55,25)

\whitebox(-45,45)
\whitebox(-45,35)
\whitebox(-45,25)
\whitebox(-45,5)
\whitebox(-45,-5)
\whitebox(-45,-15)
\whitebox(-45,-25)

\whitebox(-35,45)
\whitebox(-35,35)
\whitebox(-35,25)
\whitebox(-35,5)
\whitebox(-35,-5)
\whitebox(-35,-15)
\whitebox(-35,-25)

\whitebox(-15,35)
\whitebox(-15,25)

\rput[bl](-50,-20){\tiny{$\times$}}
\rput[bl](-40,-20){\tiny{$\times$}}
\rput[bl](-30,-20){\tiny{$\times$}}
\rput[bl](-20,-20){\tiny{$\times$}}
\rput[bl](-10,-20){\tiny{$\times$}}

\rput[bl](-40,30){\tiny{$\times$}}
\rput[bl](-30,30){\tiny{$\times$}}
\rput[bl](-20,30){\tiny{$\times$}}
\rput[bl](-10,30){\tiny{$\times$}}

\rput[bl](-30,40){\tiny{$\times$}}
\rput[bl](-20,40){\tiny{$\times$}}
\rput[bl](-10,40){\tiny{$\times$}}

\rput[bl](-20,-50){\tiny{$\times$}}
\rput[bl](-10,-50){\tiny{$\times$}}

\rput[bl](-10,-10){\tiny{$\times$}}

\psline[linewidth=1.5pt]{-}(-35,-25)(-35,-15)(-25,-15)(-25,-25)(-35,-25)
\rput(-30,-20){\footnotesize{$2$}}

\psline[linewidth=1.5pt]{-}(-35,25)(-35,35)(-25,35)(-25,25)(-35,25)
\rput(-30,30){\footnotesize{$3$}}

\rput[l](-55,-65){$w=1\;5\;\bar{4}\;\bar{3}\;2$}

\rput(-30,-97){(b)}

\endpspicture

\caption{Diagrams and essential sets of signed permutations.  Example (a) is basic; example (b) is not. \label{f.essential-ex1}}
\end{figure}

\section{Rank conditions and Bruhat order}\label{s.rank}

The main results of this section show that essential sets determine the corresponding signed permutations, and can be used to make comparisons in Bruhat order.  We also establish the connection with the general poset-theoretic essential sets of \cite{rwy}.  The statements are primarily combinatorial, but since some of the terminology and motivation is geometric, we begin by briefly reviewing the context.

Let $V$ be an odd-dimensional vector space, say $\dim V = 2n+1$, and fix a complete flag of subspaces
\[
  V = F_{\bar{n}} \supset  \cdots \supset F_{\bar{1}} \supset F_0 \supset F_1 \supset \cdots \supset F_n,
\]
with $\dim F_i = n+1-i$.  
Each permutation $v\in S_{2n+1}$ determines a Schubert variety in the complete flag variety, defined by
\[
  \Omega_v = \{ E_\bullet \,|\, \dim(E_p \cap F_q) \geq r_v(p,q) \text{ for } \bar{n}\leq p,q \leq n \}.
\]
The rank conditions in this definition are highly redundant.  Up to a change of notation, the results of \cite[\S3]{fulton} show that it is sufficient to restrict to those $(p,q)$ which are essential positions for $v$.

Now equip $V$ with a nondegenerate quadratic form, and assume $F_\bullet$ is isotropic.  A signed permutation $w\in W_n$ defines a Schubert variety in the orthogonal flag variety, by a similar prescription:
\[
  \Omega_w = \{ E_\bullet \,|\, \dim(E_p \cap F_q) \geq r_w(p,q) \text{ for } 1\leq p\leq n \text{ and }\bar{n}\leq q \leq n \}.
\]
We will see that the rank conditions $\dim(E_p\cap F_q) \geq k$ for $(k,p,q)\in\Ess(w)$ suffice to define $\Omega_w$.  (As mentioned in the introduction, the same discussion applies to degeneracy loci, with the flag varieties replaced by an arbitrary base variety.)

\subsection{Basic permutations and signed permutations}\label{ss.basic}

A basic triple $(k,p,q)$ of type A defines a \define{basic permutation} $v(k,p,q)$, which is the minimal element in Bruhat order such that $r_v(p,q)\geq k$.   
It can be written down as follows: start at position $\bar{p}$, and, proceeding {\em right to left}, place $k$ entries in descending order, ending with $q$; then fill in the remaining entries in increasing order from left to right.  
For example, $v(3,\bar{1},2) = \bar{4}\; \bar{3}\; \bar{2}\; {\bm 2}\; {\bm 3}\;{\bm 4}\;\bar{1}\; 0\; 1$.  (The three entries placed in the first step are at positions $1$, $0$, and $\bar{1}$, shown in bold.)

Recall that a permutation $v$ is \define{grassmannian} if it has exactly one descent, and $v$ is \define{bigrassmannian} if both $v$ and $v^{-1}$ are grassmannian.  From the construction, $v(k,p,q)$ is grassmannian, having a single descent at $\bar{p}$.  Moreover, $v^{-1}$ is equal to $v(k,\bar{q}+1,\bar{p}+1)$, so $v$ is bigrassmannian.  In fact, all bigrassmannian permutations arise this way: the basic permutations are exactly the bigrassmannian ones (see, e.g., \cite[Lemma~4.1]{rwy}).
  
It follows from results of Kobayashi \cite{kob} and Reiner, Woo, and Yong \cite{rwy} that for a permutation $v$, the essential set $\Ess(v)$ coincides with the set of those basic triples $(k,p,q)$ such that $v(k,p,q)$ is maximal among all basic permutations below $v$ in Bruhat order on the symmetric group.  

We will provide an analogue of this fact for signed permutations.  
Parallel to the type A situation, the \define{basic signed permutation} corresponding to a basic triple of type B is the element $w=w(k,p,q)$ which is minimal (in Bruhat order), such that $r_w(p,q) = k$.  It is not immediately obvious that such a minimum is unique, though it is not hard to prove (see Lemma~\ref{l.unique}).  To write $w(k,p,q)$, start at position $p$, and place $k$ consecutive entries, in increasing order, ending with $\bar{q}$; then fill in the unused positive entries in increasing order.  If $0<\bar{q}<k$, then ``$q-1,\; 1$'' counts as ``consecutive''.  For example, $w(2,2,3) = 1\;\bar{4}\;\bar{3}\;2$, and $w(3,2,\bar{2}) = 4\;\bar{3}\;1\;2$.

The explicit formula for a basic signed permutation depends on the relative order of $k$, $p$, $q$, and $0$.

\medskip

\renewcommand{\arraystretch}{1.3}
\begin{tabular}{| c || l |} \hline
   &  $w(k,p,q)$  \\ \hline
$q\geq p$ &   $1,\;\cdots,\; p-1,\,\;\bar{q+k-1},\;\cdots,\;\bar{q},\;p,\;\cdots,\;q-1$ \\ \hline
$p > q > 0$ & $1,\;\cdots,\; q-1,\;q+k,\;\cdots,\;p+k,\;\bar{q+k-1},\;\cdots,\;\bar{q}$ \\ \hline
$k> \bar{q} >0$ & $k+1,\;\cdots,\; p+k-1,\;\bar{k},\;\cdots,\;q-1,\;1,\;\cdots,\;\bar{q}$  \\ \hline
$\bar{q} \geq k$ & $1,\;\cdots,\; \bar{q}-k,\;\bar{q}+1,\;\cdots,\;p+k-1,\;\bar{q+k-1},\;\cdots,\;\bar{q}$ \\ \hline
\end{tabular}
\renewcommand{\arraystretch}{1.0}

\medskip

\noindent
The inverse of $w(k,p,q)$ is $w(k,q,p)$ when $q>0$, and it is $w(p+q+k-1,\bar{q-1},\bar{p-1})$ when $q<0$; in particular, the inverse of a basic signed permutation is again basic.  
The length of $w(k,p,q)$ is computed as follows:
  
\begin{center}
\renewcommand{\arraystretch}{1.3}
\begin{tabular}{| c || l |} \hline
   &  $\ell(w(k,p,q))$  \\ \hline
$q > 0$ & $(p+q-1)k + \binom{k}{2}$ \\ \hline
$k> \bar{q} >0$ & $pk + \binom{k}{2} - \binom{\bar{q}+1}{2}$  \\ \hline
$\bar{q} \geq k$ & $(p+q+k-1)k$ \\ \hline
\end{tabular}
\renewcommand{\arraystretch}{1.0}
\end{center}

\begin{remark}
A basic signed permutation $w(k,p,q)$ is grassmannian, with a unique descent at $p-1$.  Since its inverse is also basic, the basic elements are bigrassmannian.  However, in contrast to type A, not all bigrassmannians are basic.  For example, in $W_4$, there are $45$ bigrassmannians, and all but one are basic; the exception is $1\;4\;\bar{3}\;2$.  In $W_5$, there are five elements that are bigrassmannian but not basic: $1\;4\;\bar{3}\;2\;5$, $1\;4\;5\;\bar{3}\;2$, $1\;5\;\bar{4}\;\bar{3}\;2$, $1\;5\;\bar{4}\;2\;3$, and $1\;2\;5\;\bar{4}\;3$.  See \cite[\S4]{gk} for a direct comparison of bigrassmannians and basic elements in $W_n$ via generating functions for each.
\end{remark}

The assignment of a basic signed permutation to a basic triple is one-to-one.  The smallest $n$ such that $w(k,p,q) \in W_n$ is $n(k,p,q) = \max\{p+k-1,\,q+k-1\}$.  This makes it easy to enumerate the basic elements of $W_n$: we have
\begin{equation}\label{e.basic-enum}
  \#\{ w(k,p,q) \in W_n \} =  \frac{2n^3+n}{3}.
\end{equation}

We now come to the poset-theoretic characterization of the essential set.  For $v\in S_\infty$, the main theorem of \cite{kob} identifies $\Ess(v)$ with the set of type A basic triples $(k,p,q)$ such that $v(k,p,q)$ is maximal among all basic permutations below $v$ in Bruhat order.  We have an analogue of this for signed permutations:

\begin{theorem}\label{t.ess-diagram}
Let $w$ be a signed permutation.  The essential set $\Ess(w)$ is equal to the set of basic triples $(k,p,q)$ such that $w(k,p,q)$ is maximal among all basic signed permutations below $w$ in Bruhat order.  
\end{theorem}

\noindent 
In light of the theorem, when no confusion seems likely, we will identify $\Ess(w)$ with the corresponding set of basic signed permutations $w(k,p,q)$.

The proof of Theorem~\ref{t.ess-diagram} is postponed to \S\ref{s.proof}.  In the rest of this section, we elaborate on some consequences and related facts.

An important property of the essential set is that $\Ess(w)$ determines $w$.  To state this more precisely, we use the notion of supremum in a poset.  Given a poset $P$ and a subset $Y\subseteq P$, an element $x\in P$ is the \define{supremum} of $Y$ if $x \geq y$ for all $y\in Y$, and if $x' \geq y$ for all $y\in Y$, then $x'\geq x$.  (See \cite[\S2.4]{gk}.)  A supremum is clearly unique if it exists; when it does exist, write $x=\sup(Y)$.

The following is an analogue of \cite[Lemmas~3.10 and 3.14]{fulton}.

\begin{theorem}\label{t.essential-b}
Let $w$ be a signed permutation.
\begin{enumerate}
\item $w$ is the supremum of its essential set.  

\medskip

\item The essential set $\Ess(w)$ is minimal with the property that $w=\sup(\Ess(w))$; in other words, $w \neq \sup(Y)$ for any $Y\subsetneq \Ess(w)$.  More precisely, choose any $(k_0,p_0,q_0) \in \Ess(w)$.  Then one can find an element $w'\in W_n$ such that
\begin{align*}
 r_{w'}(q_0,p_0) &< k_0, \quad \text{ and } \\
r_{w'}(q,p) & \geq k \quad \text{ for all }(k,p,q) \in \Ess(w)\setminus\{(k_0,p_0,q_0)\}.
\end{align*}

\medskip

\item If $Y$ is any set of basic signed permutations such that $\sup(Y)=w$, then $Y\supseteq\Ess(w)$.
\end{enumerate}
In particular, $\Ess(w)$ is the unique minimial set of basic elements whose supremum is $w$.
\end{theorem}

\noindent
In the course of proving this, we will establish a connection with the \emph{base} of the Coxeter group $W_n$, as well as the notion of essential set considered in \cite{rwy}.  This is done in \S\ref{ss.bases}.  First, we briefly digress to describe the connection with geometry.

\subsection{Rank conditions}\label{ss.ranks}

We will use the notation $C(k,p,q)$ to denote the rank condition $\dim(E_p\cap F_q)\geq k$.  The basic triples $(k,p,q)$ of type A (defined in \S\ref{ss.essential-typeA}) correspond to rank conditions $C(k,p,q)$ which are nontrivial and for which equality is feasible; similarly, the basic triples of type B (\S\ref{ss.essential-signed}) correspond to rank conditions which are nontrivial and feasible for isotropic subspaces.

Using the correspondence between Bruhat order and containment of Schubert varieties, the fact that $v(k,p,q)$ is the minimal permutation in $S_\infty$ satisfying $r_v(p,q)\geq k$ says the Schubert variety $\Omega_{v(k,p,q)}$ is defined by the single rank condition $C(k,p,q)$ inside the complete flag variety.  A little more generally, we have the following standard fact (cf.~\cite[Lemma~4.1 and Proposition~4.6]{rwy}):

\begin{lemma}
Let $(k,p,q)$ and $(k',p',q')$ be basic triples (of type A).  The rank condition $C(k,p,q)$ implies $C(k',p',q')$ if and only if $v(k,p,q)\geq v(k',p',q')$.  Furthermore, for a general permutation $v$, the condition $C(k,p,q)$ holds on the Schubert variety $\Omega_v$ if and only if $v \geq v(k,p,q)$.
\end{lemma}

The analogous statement for signed permutations and rank conditions on isotropic flags is a consequence of the following lemma.

\begin{lemma}\label{l.unique}
For a basic triple $(k,p,q)$ of type B, the corresponding basic signed permutation $w(k,p,q)$ is the unique minimum among elements $w\in W_n$ such that $r_w(p,q) \geq k$.
\end{lemma}

Indeed, Lemma~\ref{l.unique} is equivalent to the statement: the rank condition $C(k,p,q)$ holds on $\Omega_w$ if and only if $w\geq w(k,p,q)$.  We give a proof using a geometric construction.

\begin{proof}
We claim that the subvariety of the orthogonal flag variety defined by the rank condition $C(k,p,q)$ is irreducible and has codimension equal to $\ell(w(k,p,q))$.  It is therefore equal to the corresponding Schubert variety $\Omega_{w(k,p,q)}$, so the statement about Bruhat order follows.  

To prove the claim, it is enough to show the corresponding locus in the orthogonal Grassmannian is irreducible of the correct codimension.  This is an exercise in resolving singularities.  Let $Z\subseteq X=OG(n+1-p,2n+1)$ be the (reduced) subset defined by $\dim(E_p\cap F_q)\geq k$, where $F_\bullet$ is a fixed isotropic flag.  This is the closure of the locus $Z^\circ$ where equality holds.  If $q>0$, then $F_q$ is isotropic.  In this case, let $\tilde{Z} \subset X \times Gr(k,F_q)$ be pairs $(E_p,L)$ such that $L\subseteq E_p$ and $L$ is a $k$-dimensional subspace of $F_q$.  The first projection has image $Z$, and it is an isomorphism over the open subset $Z^\circ$.  The second projection is surjective onto $Gr(k,F_q)$, with orthogonal Grassmannians $OG(n+1-p-k,2n-2k+1)$ as fibers; in particular, $\tilde{Z}$ is smooth and irreducible.  Comparing the dimensions of $\tilde{Z}$ and $X$, we find that $Z$ has codimension $(p+q-1)k + \binom{k}{2}$, which is equal to $\ell(w(k,p,q))$ as claimed.  The cases $k>\bar{q}>0$ and $\bar{q}\geq k$ are similar.
\end{proof}

Lemma~\ref{l.unique} and Theorem~\ref{t.essential-b} imply a shorter description of Schubert varieties in the orthogonal flag variety.

\begin{corollary}\label{c.essential-b}
The Schubert variety $\Omega_w$ is defined by the rank conditions $C(k,p,q)$ for basic triples $(k,p,q)$ in $\Ess(w)$, and this is a minimal list of rank conditions.  
\end{corollary}

\subsection{Bases}\label{ss.bases}

As defined by Lascoux and Sch\"utzenberger \cite{ls2}, the \define{base} of a finite Coxeter group $W$ (or any finite poset) is the set of non-identity elements $x$ which cannot be written as the supremum of any subset $Y\subseteq W$ not containing $x$.\footnote{The elements of the base are sometimes called \emph{join-irreducible} elements in the poset literature.}  This notion was developed further by Geck and Kim \cite{gk}, who identify the base for each finite Coxeter group, and also by Reading \cite{reading}.

In type A, the base of the symmetric group is the set of all bigrassmannian permutations (\cite[Th\'eor\`eme~4.4]{ls2}).  Below we will prove that the basic elements $w(k,p,q)$ form the base of $W_n$, giving a new description of the base (cf.~\cite[Theorem 4.6]{gk} for another description).

\begin{remark}\label{r.correct}
In the first statement of \cite[Th\'eor\`eme 7.4]{ls2}, it is claimed that the base of $W_n$ can also be described as those $w$ such that $\iota(w)$ is the supremum of elements $v(k,p,q)$ and $v(k,p,q)^\perp$ in $S_{2n+1}$, for some type A basic triple $(k,p,q)$.\footnote{More precisely, their claim refers to an embedding of $w$ in $S_{2n}$, but this version is equivalent.}  However, the example of $w(2,2,\bar{1}) = 3\;\bar{2}\;1$ shows that this is not the case.  The gap lies in \cite[Lemme 7.3]{ls2}, which is clarified below in the discussion following Lemma~\ref{l.implies}.
\end{remark}

A finite Coxeter group (or poset) is \define{dissective} if for every element $x$ of the base, the complement $W \setminus \{x'\,|\, x'\geq x \}$ has a unique maximal element $u$; in other words, there is a disjoint decomposition $W = \{x'\,|\,x'\geq x \} \sqcup \{u'\,|\,u' \leq u\}$.  If such a $u$ exists, we call it the \define{dissecting element} associated to $x$.  The groups $S_n$ and $W_n$ (of types A and B) are dissective, but the Weyl group of type D is not \cite{ls2,gk}.

Given a basic triple $(k,p,q)$, and an integer $n\geq n(k,p,q)$ (so that $w(k,p,q)\in W_n$), set
\[
  u(k,p,q,n) = w(n+2-p-k,\, p,\, \bar{q}+1 ) \cdot  w_\circ^{(n)},
\]
replacing ``$\bar{q}+1$'' with ``$1$'' when $q=1$.  This is the maximal element of $W_n$ with $r_u(q,p)=k-1$.  Note that although the basic signed permutation $w(k,p,q)$ does not depend on $n$, the element $u(k,p,q,n)$ necessarily does.  We will soon see that these are the dissecting elements.  First, we give another characterization of these signed permutations.

\begin{lemma}\label{l.unique2}
Let $(k,p,q)$ be a basic triple of type B, and fix an integer $n$ with $n\geq n(k,p,q) := \max\{p+k-1,\,q+k-1\}$.  The element $u(k,p,q,n)$ is the unique maximum among elements $w\in W_n$ such that $r_w(p,q) < k$.
\end{lemma}

\noindent
Equivalently, $C(k,p,q)$ fails on $\Omega_w$ if and only if $w\leq u(k,p,q,n)$.  This is an easy consequence of Lemma~\ref{l.unique}, using the involution $w_\circ^{(n)}$.

An alternative characterization of the base is given in \cite[\S2.4]{gk}: it is the set of elements $w$ which are minimal in the complement of an interval $\{u'\,|\, u'\leq u\}$, for some $u$.  In general, this $u$ is not required to be uniquely determined by $w$, but for dissective posets, there is a unique maximal such $u$ for each $w$.

\begin{proposition}\label{p.base}
The basic signed permutations in $W_n$ form the base of $W_n$.
\end{proposition}

\begin{proof}
The number of basic signed permutations is equal to the cardinality of the base (see \cite[p.~24, Remarque]{ls2} or \cite[p.~300]{gk}), so we only need to establish one inclusion.  We will show that every basic signed permutation lies in the base.

Since every element $w\in W_n$ has either $r_w(p,q)\geq k$ or $r_w(p,q)<k$, Lemmas~\ref{l.unique} and \ref{l.unique2} show that
\[
  W_n = \{ w\,|\, w\geq w(k,p,q) \} \sqcup \{ w\,|\, w\leq u(k,p,q,n) \}.
\]
In particular, $w(k,p,q)$ is the (unique) minimal element in the complement of the interval $\{ w\,|\, w\leq u(k,p,q,n) \}$, so it lies in the base.
\end{proof}

The proof also shows that $W_n$ is dissective: since every element of the base is of the form $w(k,p,q)$, the fact that $u(k,p,q,n)$ is the unique maximum in the complement of $\{ w\,|\, w\geq w(k,p,q) \}$ shows that $u(k,p,q,n)$ is the dissecting element corresponding to $w(k,p,q)$.

We now turn to Theorem~\ref{t.essential-b}.

\begin{proof}[Proof of Theorem~\ref{t.essential-b}]
In general, an element $w$ of a finite poset is equal to the supremum of the set of base elements lying below $w$ \cite[\S2.4]{gk}.  It is sufficient to include the maximal such elements, and by Theorem~\ref{t.ess-diagram} together with Proposition~\ref{p.base}, for $w\in W_n$, these constitute $\Ess(w)$; therefore $w=\sup(\Ess(w))$.  

For the second statement in (ii), concerning minimality, observe that $w(k_0,p_0,q_0) \not\leq w(k,p,q)$ for any $(k,p,q) \in \Ess(w)\setminus\{(k_0,p_0,q_0)\}$, because by Theorem~\ref{t.ess-diagram}, elements of the essential set are incomparable.  By the dissective property, it follows that $w(k,p,q) \leq u(k_0,p_0,q_0,n)$ for all such $(k,p,q)$, and therefore the intersection
\[
  \{ w\,|\, w\leq u(k_0,p_0,q_0,n) \} \cap \bigcap_{(k,p,q)\in\Ess(w)\setminus\{(k_0,p_0,q_0)\}} \{ w\,|\, w\geq w(k,p,q) \}
\]
is nonempty.  Any element of this intersection can be taken as the desired $w'$; for instance, $w'=u(k_0,p_0,q_0,n)$ works.
\end{proof}

\subsection{A bijection}\label{ss.biject}

Finally, we make the relationship with the essential set of \cite{rwy} precise, for both permutations and signed permutations.

For the purposes of this subsection only, let us consider permutations in $S_n$ as acting on $\{1,\ldots,n\}$ in the usual way; the definitions and conventions for signed permutations in $W_n$ remain as before.  This results in some changes of notation for type A, but no further complications.  The rank function for $v \in S_n$ is redefined as $r_v(p,q) = \#\{ i\leq p\,|\, v(i) > q\}$, and basic triples of type A have $k,p,q>0$ and $p\geq k > p-q$.  In one-line notation, the corresponding basic permutation $v(k,p,q)$ is
\[
  v(k,p,q) = 1,\;\cdots,\;p-k,\;q+1,\;\cdots\;q+k,\;p-k+1,\;\cdots,\;q.
\]
For example, $v(3,4,2) = 1\;3\;4\;5\;2$.

As mentioned before, basic permutations are bigrassmannian, using the fact that the inverse of $v(k,p,q)$ is equal to $v(p+q-k,q,p)$, and they account for all bigrassmannians in $S_n$.

The assignment of $v(k,p,q)$ to a basic triple is one-to-one.  It is easy to describe the inverse map: Given a bigrassmannian permutation, record the position of the unique descent as $p$; starting at this position, collect all the consecutive entries to the left of it---let $k$ be the number of them, and let $q+1$ be the smallest one.  Furthermore, the smallest $n$ such that $v(k,p,q)$ lies in $S_n$ is $q+k$.  These two facts make it easy to enumerate the basic permutations in $S_n$ (or equivalently, the bigrassmannians)---there are $\frac{1}{6}(n^3-n) = \binom{n+1}{3}$ of them.

We will need the dissecting elements for $S_n$, which are defined similarly to the elements $u(k,p,q,n)$ for $W_n$.  Given a basic triple (of type A) $(k,p,q)$ and $n\geq q+k$, so $v(k,p,q) \in S_n$, set $t(k,p,q,n) = v(n+1-q-k, n-p, q)\cdot w_\circ^{(n)}$.  (This is the maximal element in $S_n$ having $r_t(q,p) < k$.)  

The diagram of a permutation in $S_n$ is defined as usual, via the permutation matrix with dots in positions $(v(i),i)$.  The essential set of $v \in S_n$ is the set of basic triples $(k,p,q)$ such that $(q,p)$ is a SE corner of the diagram of $v$ and $k=r_v(p,q)$.

Given a permutation $v\in S_n$, let $\RWY(v)$ denote the set of all permutations $t\in S_n$ which are minimal (in Bruhat order) among those not below $v$.  For $w\in W_n$, define $\RWY(w)$ analogously.  We write $w_\circ^{(n)}$ for the longest element of $S_n$ or $W_n$, depending on context.

\begin{proposition}\label{p.rwy-equivalent}
Fix $n$.
\begin{enumerate}
\item For any $v\in S_n$, the map sending $(k,p,q) \in \Ess(vw_\circ^{(n)})$ to $t(k,p,q,n)\cdot w_\circ^{(n)}$ is a bijection onto $\RWY(v)$.

\smallskip

\item For any $w\in W_n$, the map sending $(k,p,q) \in \Ess(ww_\circ^{(n)})$ to $u(k,p,q,n)\cdot w_\circ^{(n)}$ is a bijection onto $\RWY(w)$.
\end{enumerate}
\end{proposition}

The proof is a translation of the argument in \cite[Proposition~4.6]{rwy}.  The factor of $w_\circ$ arises from the fact that their Schubert varieties $X_v$ correspond to our $\Omega_{vw_\circ}$; furthermore, our $E_p$ is their $(\C^n/V_{n-p})^*$.  (Combining these two identifications explains why $X_v$ appears to be identified with $\Omega_{w_\circ v}$ in \cite{rwy}.)  Similarly, their $X_w$ is our $\Omega_{ww_\circ}$ in the orthogonal flag variety.

\begin{example}
To illustrate, we compare with \cite[Example 4.7]{rwy}.  Here $n=6$, and $v=4\;2\;5\;1\;6\;3$.  They show that
\[
  \RWY( 4\;2\;5\;1\;6\;3 ) = \{ 1\;2\;3\;6\;4\;5, \;  1\;3\;4\;5\;2\;6, \; 1\;5\;2\;3\;4\;6, \; 3\;4\;1\;2\;5\;6 \}.
\]
Now $vw_\circ = 3\;6\;1\;5\;2\;4$, and one can compute
\[
  \Ess( 3\;6\;1\;5\;2\;4 ) = \{ (1,2,5), \; (2,2,2),\; (2,4,4),\; (3,4,2) \}.
\]
The corresponding dissecting elements $t(k,p,q,6)$ are
\[
  5\;4\;6\;3\;2\;1, \; 6\;2\;5\;4\;3\;1, \; 6\;4\;3\;2\;5\;1, \; 6\;5\;2\;1\;4\;3,
\]
so one recovers $\RWY(v)$ by applying $w_\circ^{(6)}$.
\end{example}

\section{Proof of Theorem~\ref{t.ess-diagram}}\label{s.proof}

Some temporary notation will be useful.  Let $\mathscr{B}$ be the set of all type B basic triples.  For any $w\in W_\infty$, let $\Max(w)$ be the set of type B basic triples $(k,p,q)$ such that $w(k,p,q)$ is maximal among all basic signed permutations below $w$ in Bruhat order.  Thus we must show $\Max(w)=\Ess(w)$.

Defining $\Max(v)$ analogously for $v\in S_\infty$, we have $\Max(\iota(w)) = \Ess(\iota(w))$ by \cite{kob} (or \cite[Proposition 4.6]{rwy}).  From the definitions, $\Ess(w)$ is contained in $\Ess(\iota(w)) \cap \mathscr{B}$.  To prove the theorem, we will show that $\Max(w)$ is also contained in $\Ess(\iota(w)) \cap \mathscr{B}$, and then show that the exceptions are the same as the ones specified in the definition of $\Ess(w)$ (Definition~\ref{d.essential}).  The key ingredient is a comparison of the Bruhat orders on $S_\infty$ and $W_\infty$, so we begin with two lemmas concerning this.  Throughout, we implicitly use the fact that $w\leq w'$ in $W_\infty$ iff $\iota(w)\leq \iota(w')$ in $S_\infty$---see, e.g., \cite[Corollary~8.1.9]{bb}.  (The ``if'' statement fails for the type D Weyl group, giving another hint that this case is more complicated.)

Suppose $(k,p,q)$ and $(k',p',q')$ are two basic triples of type B, hence also of type A.  It is easy to see that if $v(k',p',q') \geq v(k,p,q)$ in type $S_\infty$, then $w(k',p',q') \geq w(k,p,q)$ in $W_\infty$.  The converse is not true, but the exceptions are easily classified.

\begin{lemma}\label{l.implies}
Suppose $w(k',p',q') \geq w(k,p,q)$ in type $W_\infty$, but $v(k',p',q') \not\geq v(k,p,q)$ in $S_\infty$.  Then either
\begin{enumerate}
\item $-k'< q' < 0 < q$, and $v(k'+q',p',\bar{q'}+1) \geq v(k,p,q)$ in $S_\infty$; or

\smallskip

\item $q' > 0 > q$, and $v(k',p',q')^\perp \geq v(k,p,q)$ in $S_\infty$.
\end{enumerate}
\end{lemma}

\begin{proof}
We have $v(k',p',q') \geq v(k,p,q)$ in $S_\infty$ (respectively, $w(k',p',q') \geq w(k,p,q)$ in $W_\infty$) exactly when $r_{v(k',p',q')}(p,q) \geq k$ (resp., $r_{w(k',p',q')}(p,q) \geq k$).  To prove the lemma, therefore, we need to compare the rank functions of $v(k',p',q')$ and $w(k',p',q')$.  The hypothesis is that
\begin{equation}\label{e.except}
  r_{w(k',p',q')}(p,q) \geq k > r_{v(k',p',q')}(p,q).
\end{equation}

When $q'\leq -k'$, the negative columns of the permutation matrices for $w(k',p',q')$ and $v(k',p',q')$ are the same, and so the corresponding rank functions agree; the inequality \eqref{e.except} cannot hold in this case.

Consider $-k'<q'<0$.  In this case the rank functions for $w=w(k',p',q')$ and $v=v(k',p',q')$ agree outside the region where $0 < p,q \leq k'+q'+1$ and $p+q \leq p'+k'$.  Inside this region, $r_w(p,q)>r_v(p,q)$, so one can choose $k$ so that \eqref{e.except} holds.  However, in this region $r_w(p,q) = r_{v(k'+q',p',\bar{q'}+1)}(p,q)$, so the lemma is true in this case.  (See Figure~\ref{f.case2} for an illustration.)

Finally, assume $q'>0$.  As in the previous case, the rank functions for $w=w(k',p',q')$ and $v=v(k',p',q')$ agree outside the region where $0<p<q'$ and $-k'-q'+1<q\leq-p$.  Inside this region, $r_w(p,q)>r_v(p,q)$, but $r_w(p,q) = r_{v(k',p',q')^\perp}(p,q)$, so the lemma is true in this case as well.  (See Figure~\ref{f.case3}.)
\end{proof}


\begin{figure}[ht]
\pspicture(70,70)(-100,-100)

\psset{unit=1.35pt}

\pspolygon[fillstyle=solid,fillcolor=lightgray,linecolor=lightgray](-45,5)(-5,5)(-5,-35)(-25,-35)(-25,-25)(-35,-25)(-35,-15)(-45,-15)(-45,5)

\psline{-}(-5,55)(-55,55)
\psline{-}(-5,45)(-55,45)
\psline{-}(-5,35)(-55,35)
\psline{-}(-5,25)(-55,25)
\psline{-}(-5,15)(-55,15)
\psline{-}(-5,5)(-55,5)
\psline{-}(-5,-5)(-55,-5)
\psline{-}(-5,-15)(-55,-15)
\psline{-}(-5,-25)(-55,-25)
\psline{-}(-5,-35)(-55,-35)
\psline{-}(-5,-45)(-55,-45)
\psline{-}(-5,-55)(-55,-55)

\psline{-}(-5,55)(-5,-55)
\psline{-}(-15,55)(-15,-55)
\psline{-}(-25,55)(-25,-55)
\psline{-}(-35,55)(-35,-55)
\psline{-}(-45,55)(-45,-55)
\psline{-}(-55,55)(-55,-55)

\rput(-65,50){$\bar{5}$}
\rput(-65,40){$\bar{4}$}
\rput(-65,30){$\bar{3}$}
\rput(-65,20){$\bar{2}$}
\rput(-65,10){$\bar{1}$}
\rput(-65,0){$0$}
\rput(-65,-10){$1$}
\rput(-65,-20){$2$}
\rput(-65,-30){$3$}
\rput(-65,-40){$4$}
\rput(-65,-50){$5$}

\pscircle*(-50,10){2}
\pscircle*(-40,-20){2}
\pscircle*(-30,-30){2}
\pscircle*(-20,-40){2}
\pscircle*(-10,50){2}

\psline[linewidth=1.5pt]{-}(-25,15)(-25,25)(-15,25)(-15,15)(-25,15)

\rput(-20,20){\footnotesize{$4$}}

\rput(-40,0){\footnotesize{$1$}}
\rput(-40,-10){\footnotesize{$1$}}

\rput(-30,0){\footnotesize{$2$}}
\rput(-30,-10){\footnotesize{$2$}}
\rput(-30,-20){\footnotesize{$1$}}

\rput(-20,0){\footnotesize{$3$}}
\rput(-20,-10){\footnotesize{$3$}}
\rput(-20,-20){\footnotesize{$2$}}
\rput(-20,-30){\footnotesize{$1$}}

\rput(-10,0){\footnotesize{$3$}}
\rput(-10,-10){\footnotesize{$3$}}
\rput(-10,-20){\footnotesize{$2$}}
\rput(-10,-30){\footnotesize{$1$}}

\rput[l](-45,-65){$w(4,2,\bar{1})$}

\endpspicture
\pspicture(70,70)(-30,-100)

\psset{unit=1.35pt}

\pspolygon[fillstyle=solid,fillcolor=lightgray,linecolor=lightgray](-45,5)(-5,5)(-5,-35)(-25,-35)(-25,-25)(-35,-25)(-35,-15)(-45,-15)(-45,5)

\psline{-}(-5,55)(-55,55)
\psline{-}(-5,45)(-55,45)
\psline{-}(-5,35)(-55,35)
\psline{-}(-5,25)(-55,25)
\psline{-}(-5,15)(-55,15)
\psline{-}(-5,5)(-55,5)
\psline{-}(-5,-5)(-55,-5)
\psline{-}(-5,-15)(-55,-15)
\psline{-}(-5,-25)(-55,-25)
\psline{-}(-5,-35)(-55,-35)
\psline{-}(-5,-45)(-55,-45)
\psline{-}(-5,-55)(-55,-55)

\psline{-}(-5,55)(-5,-55)
\psline{-}(-15,55)(-15,-55)
\psline{-}(-25,55)(-25,-55)
\psline{-}(-35,55)(-35,-55)
\psline{-}(-45,55)(-45,-55)
\psline{-}(-55,55)(-55,-55)

\pscircle*(-50,10){2}
\pscircle*(-40,0){2}
\pscircle*(-30,-10){2}
\pscircle*(-20,-20){2}
\pscircle*(-10,50){2}

\psline[linewidth=1.5pt]{-}(-25,15)(-25,25)(-15,25)(-15,15)(-25,15)

\rput(-20,20){\footnotesize{$4$}}

\rput(-40,-10){\footnotesize{$0$}}

\rput(-30,0){\footnotesize{$1$}}
\rput(-30,-20){\footnotesize{$0$}}

\rput(-20,0){\footnotesize{$2$}}
\rput(-20,-10){\footnotesize{$1$}}
\rput(-20,-30){\footnotesize{$0$}}

\rput(-10,0){\footnotesize{$2$}}
\rput(-10,-10){\footnotesize{$1$}}
\rput(-10,-20){\footnotesize{$0$}}
\rput(-10,-30){\footnotesize{$0$}}

\rput[l](-45,-65){$v(4,2,\bar{1})$}

\endpspicture
\pspicture(70,70)(-30,-100)

\psset{unit=1.35pt}

\pspolygon[fillstyle=solid,fillcolor=lightgray,linecolor=lightgray](-45,5)(-5,5)(-5,-35)(-25,-35)(-25,-25)(-35,-25)(-35,-15)(-45,-15)(-45,5)

\psline{-}(-5,55)(-55,55)
\psline{-}(-5,45)(-55,45)
\psline{-}(-5,35)(-55,35)
\psline{-}(-5,25)(-55,25)
\psline{-}(-5,15)(-55,15)
\psline{-}(-5,5)(-55,5)
\psline{-}(-5,-5)(-55,-5)
\psline{-}(-5,-15)(-55,-15)
\psline{-}(-5,-25)(-55,-25)
\psline{-}(-5,-35)(-55,-35)
\psline{-}(-5,-45)(-55,-45)
\psline{-}(-5,-55)(-55,-55)

\psline{-}(-5,55)(-5,-55)
\psline{-}(-15,55)(-15,-55)
\psline{-}(-25,55)(-25,-55)
\psline{-}(-35,55)(-35,-55)
\psline{-}(-45,55)(-45,-55)
\psline{-}(-55,55)(-55,-55)

\pscircle*(-50,50){2}
\pscircle*(-40,-20){2}
\pscircle*(-30,-30){2}
\pscircle*(-20,-40){2}
\pscircle*(-10,40){2}

\psline[linewidth=1.5pt]{-}(-25,-15)(-25,-5)(-15,-5)(-15,-15)(-25,-15)

\rput(-20,-10){\footnotesize{$3$}}

\rput(-40,0){\footnotesize{$1$}}
\rput(-40,-10){\footnotesize{$1$}}

\rput(-30,0){\footnotesize{$2$}}
\rput(-30,-10){\footnotesize{$2$}}
\rput(-30,-20){\footnotesize{$1$}}

\rput(-20,0){\footnotesize{$3$}}
\rput(-20,-20){\footnotesize{$2$}}
\rput(-20,-30){\footnotesize{$1$}}

\rput(-10,0){\footnotesize{$3$}}
\rput(-10,-10){\footnotesize{$3$}}
\rput(-10,-20){\footnotesize{$2$}}
\rput(-10,-30){\footnotesize{$1$}}

\rput[l](-45,-65){$v(3,2,2)$}

\endpspicture

\caption{\label{f.case2}}
\end{figure}


\begin{figure}[h]
\pspicture(70,90)(-100,-100)

\psset{unit=1.35pt}

\pspolygon[fillstyle=solid,fillcolor=lightgray,linecolor=lightgray](-25,55)(-5,55)(-5,15)(-15,15)(-15,25)(-25,25)(-25,55)

\psline{-}(-5,65)(-65,65)
\psline{-}(-5,55)(-65,55)
\psline{-}(-5,45)(-65,45)
\psline{-}(-5,35)(-65,35)
\psline{-}(-5,25)(-65,25)
\psline{-}(-5,15)(-65,15)
\psline{-}(-5,5)(-65,5)
\psline{-}(-5,-5)(-65,-5)
\psline{-}(-5,-15)(-65,-15)
\psline{-}(-5,-25)(-65,-25)
\psline{-}(-5,-35)(-65,-35)
\psline{-}(-5,-45)(-65,-45)
\psline{-}(-5,-55)(-65,-55)
\psline{-}(-5,-65)(-65,-65)

\psline{-}(-5,65)(-5,-65)
\psline{-}(-15,65)(-15,-65)
\psline{-}(-25,65)(-25,-65)
\psline{-}(-35,65)(-35,-65)
\psline{-}(-45,65)(-45,-65)
\psline{-}(-55,65)(-55,-65)
\psline{-}(-65,65)(-65,-65)

\rput(-75,60){$\bar{6}$}
\rput(-75,50){$\bar{5}$}
\rput(-75,40){$\bar{4}$}
\rput(-75,30){$\bar{3}$}
\rput(-75,20){$\bar{2}$}
\rput(-75,10){$\bar{1}$}
\rput(-75,0){$0$}
\rput(-75,-10){$1$}
\rput(-75,-20){$2$}
\rput(-75,-30){$3$}
\rput(-75,-40){$4$}
\rput(-75,-50){$5$}
\rput(-75,-60){$6$}

\pscircle*(-60,-30){2}
\pscircle*(-50,-40){2}
\pscircle*(-40,-50){2}
\pscircle*(-30,60){2}
\pscircle*(-20,20){2}
\pscircle*(-10,10){2}

\psline[linewidth=1.5pt]{-}(-45,-25)(-45,-15)(-35,-15)(-35,-25)(-45,-25)

\rput(-40,-20){\footnotesize{$3$}}

\rput(-20,50){\footnotesize{$4$}}
\rput(-20,40){\footnotesize{$4$}}
\rput(-20,30){\footnotesize{$4$}}

\rput(-10,50){\footnotesize{$5$}}
\rput(-10,40){\footnotesize{$5$}}
\rput(-10,30){\footnotesize{$5$}}
\rput(-10,20){\footnotesize{$4$}}

\rput[l](-45,-75){$w(3,4,3)$}

\endpspicture
\pspicture(70,90)(-100,-100)

\psset{unit=1.35pt}

\pspolygon[fillstyle=solid,fillcolor=lightgray,linecolor=lightgray](-25,55)(-5,55)(-5,15)(-15,15)(-15,25)(-25,25)(-25,55)

\psline{-}(-5,65)(-65,65)
\psline{-}(-5,55)(-65,55)
\psline{-}(-5,45)(-65,45)
\psline{-}(-5,35)(-65,35)
\psline{-}(-5,25)(-65,25)
\psline{-}(-5,15)(-65,15)
\psline{-}(-5,5)(-65,5)
\psline{-}(-5,-5)(-65,-5)
\psline{-}(-5,-15)(-65,-15)
\psline{-}(-5,-25)(-65,-25)
\psline{-}(-5,-35)(-65,-35)
\psline{-}(-5,-45)(-65,-45)
\psline{-}(-5,-55)(-65,-55)
\psline{-}(-5,-65)(-65,-65)

\psline{-}(-5,65)(-5,-65)
\psline{-}(-15,65)(-15,-65)
\psline{-}(-25,65)(-25,-65)
\psline{-}(-35,65)(-35,-65)
\psline{-}(-45,65)(-45,-65)
\psline{-}(-55,65)(-55,-65)
\psline{-}(-65,65)(-65,-65)

\pscircle*(-60,-30){2}
\pscircle*(-50,-40){2}
\pscircle*(-40,-50){2}
\pscircle*(-30,60){2}
\pscircle*(-20,50){2}
\pscircle*(-10,40){2}

\psline[linewidth=1.5pt]{-}(-45,-25)(-45,-15)(-35,-15)(-35,-25)(-45,-25)

\rput(-40,-20){\footnotesize{$3$}}

\rput(-20,40){\footnotesize{$3$}}
\rput(-20,30){\footnotesize{$3$}}

\rput(-10,50){\footnotesize{$4$}}
\rput(-10,30){\footnotesize{$3$}}
\rput(-10,20){\footnotesize{$3$}}

\rput[l](-45,-75){$v(3,4,3)$}

\endpspicture

\pspicture(100,100)(-50,-100)

\psset{unit=1.35pt}

\pspolygon[fillstyle=solid,fillcolor=lightgray,linecolor=lightgray](-25,55)(-5,55)(-5,15)(-15,15)(-15,25)(-25,25)(-25,55)

\psline{-}(65,65)(-65,65)
\psline{-}(65,55)(-65,55)
\psline{-}(65,45)(-65,45)
\psline{-}(65,35)(-65,35)
\psline{-}(65,25)(-65,25)
\psline{-}(65,15)(-65,15)
\psline{-}(65,5)(-65,5)
\psline{-}(65,-5)(-65,-5)
\psline{-}(65,-15)(-65,-15)
\psline{-}(65,-25)(-65,-25)
\psline{-}(65,-35)(-65,-35)
\psline{-}(65,-45)(-65,-45)
\psline{-}(65,-55)(-65,-55)
\psline{-}(65,-65)(-65,-65)

\psline{-}(65,65)(65,-65)
\psline{-}(55,65)(55,-65)
\psline{-}(45,65)(45,-65)
\psline{-}(35,65)(35,-65)
\psline{-}(25,65)(25,-65)
\psline{-}(15,65)(15,-65)
\psline{-}(5,65)(5,-65)
\psline{-}(-5,65)(-5,-65)
\psline{-}(-15,65)(-15,-65)
\psline{-}(-25,65)(-25,-65)
\psline{-}(-35,65)(-35,-65)
\psline{-}(-45,65)(-45,-65)
\psline{-}(-55,65)(-55,-65)
\psline{-}(-65,65)(-65,-65)

\rput(-75,60){$\bar{6}$}
\rput(-75,50){$\bar{5}$}
\rput(-75,40){$\bar{4}$}
\rput(-75,30){$\bar{3}$}
\rput(-75,20){$\bar{2}$}
\rput(-75,10){$\bar{1}$}
\rput(-75,0){$0$}
\rput(-75,-10){$1$}
\rput(-75,-20){$2$}
\rput(-75,-30){$3$}
\rput(-75,-40){$4$}
\rput(-75,-50){$5$}
\rput(-75,-60){$6$}

\pscircle*(-60,60){2}
\pscircle*(-50,20){2}
\pscircle*(-40,10){2}
\pscircle*(-30,0){2}
\pscircle*(-20,-10){2}
\pscircle*(-10,-20){2}
\pscircle*(0,-30){2}
\pscircle*(10,-40){2}
\pscircle*(20,-50){2}
\pscircle*(30,-60){2}
\pscircle*(40,50){2}
\pscircle*(50,40){2}
\pscircle*(60,30){2}

\psline[linewidth=1.5pt]{-}(35,35)(35,25)(25,25)(25,35)(35,35)

\rput(30,30){\footnotesize{$9$}}

\rput(-20,50){\footnotesize{$4$}}
\rput(-20,40){\footnotesize{$4$}}
\rput(-20,30){\footnotesize{$4$}}

\rput(-10,50){\footnotesize{$5$}}
\rput(-10,40){\footnotesize{$5$}}
\rput(-10,30){\footnotesize{$5$}}
\rput(-10,20){\footnotesize{$4$}}

\rput[l](-15,-75){$v(9,\bar{3},\bar{2})$}

\endpspicture

\caption{\label{f.case3}}
\end{figure}


The above lemma and its proof show that $w(k,p,q) \leq w(k',p',q')$ in $W_\infty$ if and only if
\begin{enumerate}
\item $v(k,p,q) \leq v(k'+q',p',\bar{q'}+1)$, 
\item $v(k,p,q) \leq v(k',p',q')^\perp$, or
\item $v(k,p,q) \leq v(k',p',q')$
\end{enumerate}
in $S_\infty$.  (This clarifies the claim before \cite[Lemme~7.3]{ls2}, where the first case is missing.)

\begin{lemma}\label{l.essential-redundant}
Given a signed permutation $w$, suppose $(k,p,q)$ and $(k',p',q')$ are basic triples of type B, and are in the (type A) essential set $\Ess(\iota(w))$.  Assume that $w(k',p',q')\geq w(k,p,q)$ in $W_n$. 
Then
\begin{enumerate}
\item $p=p'$;
\item $q' < 0$ and $q = \bar{q'}+1$; and 
\item $k=k'+q'$.
\end{enumerate}
\end{lemma}

\begin{proof}
By assumption, $w(k',p',q') \geq w(k,p,q)$ in $W_\infty$, but $v(k',p',q') \not\geq v(k,p,q)$ in $S_\infty$, since elements of an essential set are incomparable.  Therefore Lemma~\ref{l.implies} applies.  In each case of that lemma, a (type A) basic triple is asserted to be greater than $(k,p,q)$; again because distinct elements of the essential set are incomparable, the asserted inequality must actually be an equality.  In the first case, this yields the above conclusion.  The second case of Lemma~\ref{l.implies} cannot hold, because $(k',p',q')^\perp = (p'+q'+k'-1,\bar{p'}+1,\bar{q'}+1)$ is not a type B basic triple.
\end{proof}

Finally, we turn to the proof of the theorem.

\begin{proof}[Proof of Theorem~\ref{t.ess-diagram}]
First we show that $\Max(w)$ is contained in the set of type B basic triples inside $\Ess(\iota(w))$.  Every basic triple in $\Ess(\iota(w))$ is either a type B basic triple or a reflection of one, in the sense of Lemma~\ref{l.symmetric-essential}; basic triples with $p>0$ and $q=0$ cannot occur because $\iota(w)(0)=0$, and triples with $p\leq 0$ and $q=1$ cannot occur for the same reason.  Now suppose $(k,p,q)$ is in $\Max(w)$, but not in $\Ess(\iota(w))$.  This means there is a (type A) basic triple $(k',p',q')$ such that $v(k,p,q) < v(k',p',q') \leq \iota(w)$.  Such a basic triple cannot be of type B, since otherwise it would violate maximality of $w(k,p,q)$, so $(k',p',q')^\perp$ must be of type $B$.  Note also that $(k',p',q')^\perp \neq (k,p,q)$, since the permutations $v(k,p,q)$ and $v(k,p,q)^\perp$ are never comparable.  Finally, since the Schubert conditions $C(k',p',q')$ and $C(k',p',q')^\perp$ are equivalent on the type B flag variety, and the first one implies $C(k,p,q)$, so the second one does as well; therefore we must have $w(k',p',q')^\perp > w(k,p,q)$, contradicting maximality again.  

Given that $\Max(w)\subseteq \Ess(\iota(w)) \cap \mathscr{B}$, the proposition is proved by determining when two type B basic triples in $\Ess(\iota(w))$ yield comparable basic signed permutations, and hence redundant conditions in type B.  This is done by Lemma~\ref{l.essential-redundant}.
\end{proof}

\section{A variation for type C}\label{s.typeC}

The group $W_n$ of signed permutations is isomorphic to the (type C) Weyl group of $Sp_{2n}$, as well as to the (type B) Weyl group of $SO_{2n+1}$.  The main results above are intrinsic to $W_n$, and so do not depend on whether one regards it as type B or type C, but it is sometimes useful to have diagrams adapted to the type C context.  Here we briefly sketch the modifications; all the results and proofs can easily be reformulated.

We will use the embdedding of Weyl groups $\iota'\colon W_n\hookrightarrow S_{2n}$, which omits the value $w(0)=0$.  For a signed permutation $w\in W_n$, one forms a matrix by placing dots in a $2n\times n$ array of boxes as before: labelling the rows $\{\bar{n},\ldots,\bar{1},1,\ldots,n\}$ and the columns $\{\bar{n},\ldots,\bar{1}\}$, dots are placed in the boxes $(w(i),i)$ for $\bar{n}\leq i\leq \bar{1}$ (or equivalently, $(\bar{w(i)},\bar\imath)$ for $1\leq i\leq n$).  An ``$\times$'' is placed in boxes $(a,b)$ such that $a=\bar{w(i)}$ for some $i<b$; that is, for each dot, place an $\times$ in the columns strictly to the right, and rows opposite the dot.  This should be compared with the parametrization of Schubert cells in the type C flag variety $Sp_{2n}/B$ given in \cite[\S6.1]{fp}.

The \define{type C extended diagram} is the collection of boxes $D_C^+(w)$ in the $2n\times n$ rectangle which remain after striking out those to the right or below a dot; the \define{type C diagram} is the subset $D_C(w)\subseteq D_C^+(w)$ of boxes not marked by an $\times$.  As before, the number of boxes in $D_C(w)$ is equal to $\ell(w)$---the boxes $(i,j)$ of $D_C(w)$ correspond to inversions $\ee_i-\ee_{w(j)}$ of $w$, using $2\ee_1,\ee_2-\ee_1,\ldots,\ee_n-\ee_{n-1}$ as the simple roots for type $\mathrm{C}_n$.  

The essential set is determined exactly as in Definition~\ref{d.essential}, except that when $q=1$, we understand $q-1$ to mean $\bar{1}$.  (Thus if $D^+_C(w)$ has a SE corner in position $(\bar{1},\bar{2})$, it determines an essential position $(p,q)=(2,1)$.)  Examples are shown in Figure~\ref{f.essentialC}.

\begin{figure}[ht]
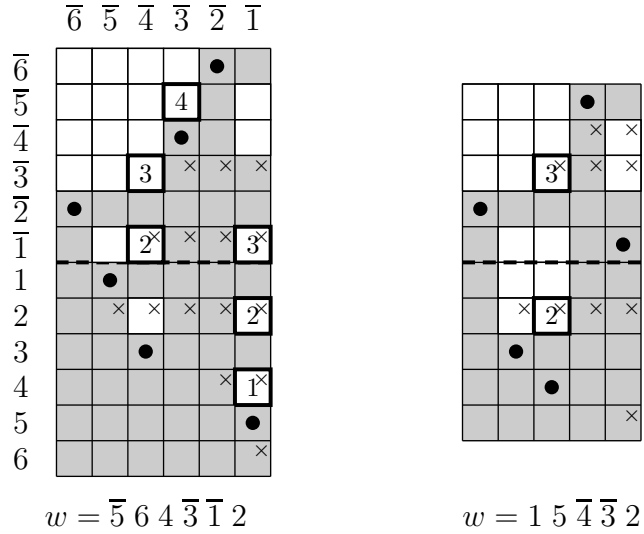

\pspicture(100,80)(-100,-90)

\psset{unit=1.35pt}

\pspolygon[fillstyle=solid,fillcolor=lightgray,linecolor=lightgray](-65,65)(-5,65)(-5,-55)(-65,-55)(-65,55)

\psline{-}(-5,65)(-65,65)
\psline{-}(-5,55)(-65,55)
\psline{-}(-5,45)(-65,45)
\psline{-}(-5,35)(-65,35)
\psline{-}(-5,25)(-65,25)
\psline{-}(-5,15)(-65,15)
\psline{-}(-5,5)(-65,5)
\psline{-}(-5,-5)(-65,-5)
\psline{-}(-5,-15)(-65,-15)
\psline{-}(-5,-25)(-65,-25)
\psline{-}(-5,-35)(-65,-35)
\psline{-}(-5,-45)(-65,-45)
\psline{-}(-5,-55)(-65,-55)

\psline{-}(-5,65)(-5,-55)
\psline{-}(-15,65)(-15,-55)
\psline{-}(-25,65)(-25,-55)
\psline{-}(-35,65)(-35,-55)
\psline{-}(-45,65)(-45,-55)
\psline{-}(-55,65)(-55,-55)
\psline{-}(-65,65)(-65,-55)

\rput(-75,60){$\bar{6}$}
\rput(-75,50){$\bar{5}$}
\rput(-75,40){$\bar{4}$}
\rput(-75,30){$\bar{3}$}
\rput(-75,20){$\bar{2}$}
\rput(-75,10){$\bar{1}$}
\rput(-75,0){$1$}
\rput(-75,-10){$2$}
\rput(-75,-20){$3$}
\rput(-75,-30){$4$}
\rput(-75,-40){$5$}
\rput(-75,-50){$6$}

\rput[b](-60,70){$\bar{6}$}
\rput[b](-50,70){$\bar{5}$}
\rput[b](-40,70){$\bar{4}$}
\rput[b](-30,70){$\bar{3}$}
\rput[b](-20,70){$\bar{2}$}
\rput[b](-10,70){$\bar{1}$}

\psline[linestyle=dashed,linewidth=1.5pt]{-}(-5,5)(-65,5)

\pscircle*(-10,-40){2}
\pscircle*(-20,60){2}
\pscircle*(-30,40){2}
\pscircle*(-40,-20){2}
\pscircle*(-50,0){2}
\pscircle*(-60,20){2}

\whitebox(-65,55)
\whitebox(-65,45)
\whitebox(-65,35)
\whitebox(-65,25)

\whitebox(-55,55)
\whitebox(-55,45)
\whitebox(-55,35)
\whitebox(-55,25)
\whitebox(-55,5)

\whitebox(-45,55)
\whitebox(-45,45)
\whitebox(-45,35)
\whitebox(-45,25)
\whitebox(-45,5)
\whitebox(-45,-15)

\whitebox(-35,55)
\whitebox(-35,45)

\whitebox(-15,45)
\whitebox(-15,35)
\whitebox(-15,5)
\whitebox(-15,-15)
\whitebox(-15,-35)

\rput[bl](-50,-10){\tiny{$\times$}}
\rput[bl](-40,-10){\tiny{$\times$}}
\rput[bl](-30,-10){\tiny{$\times$}}
\rput[bl](-20,-10){\tiny{$\times$}}
\rput[bl](-10,-10){\tiny{$\times$}}

\rput[bl](-40,10){\tiny{$\times$}}
\rput[bl](-30,10){\tiny{$\times$}}
\rput[bl](-20,10){\tiny{$\times$}}
\rput[bl](-10,10){\tiny{$\times$}}

\rput[bl](-30,30){\tiny{$\times$}}
\rput[bl](-20,30){\tiny{$\times$}}
\rput[bl](-10,30){\tiny{$\times$}}

\rput[bl](-20,-30){\tiny{$\times$}}
\rput[bl](-10,-30){\tiny{$\times$}}

\rput[bl](-10,-50){\tiny{$\times$}}

\psline[linewidth=1.5pt]{-}(-45,5)(-45,15)(-35,15)(-35,5)(-45,5)
\rput(-40,10){\footnotesize{$2$}}

\psline[linewidth=1.5pt]{-}(-45,25)(-45,35)(-35,35)(-35,25)(-45,25)
\rput(-40,30){\footnotesize{$3$}}

\psline[linewidth=1.5pt]{-}(-35,45)(-35,55)(-25,55)(-25,45)(-35,45)
\rput(-30,50){\footnotesize{$4$}}

\psline[linewidth=1.5pt]{-}(-15,-35)(-15,-25)(-5,-25)(-5,-35)(-15,-35)
\rput(-10,-30){\footnotesize{$1$}}

\psline[linewidth=1.5pt]{-}(-15,-15)(-15,-5)(-5,-5)(-5,-15)(-15,-15)
\rput(-10,-10){\footnotesize{$2$}}

\psline[linewidth=1.5pt]{-}(-15,5)(-15,15)(-5,15)(-5,5)(-15,5)
\rput(-10,10){\footnotesize{$3$}}

\rput(-40,-65){$w=\bar{5}\;6\;4\;\bar{3}\;\bar{1}\;2$}

\endpspicture
\pspicture(70,110)(-40,-90)

\psset{unit=1.35pt}

\pspolygon[fillstyle=solid,fillcolor=lightgray,linecolor=lightgray](-55,55)(-5,55)(-5,-45)(-55,-45)(-55,55)

\psline{-}(-5,55)(-55,55)
\psline{-}(-5,45)(-55,45)
\psline{-}(-5,35)(-55,35)
\psline{-}(-5,25)(-55,25)
\psline{-}(-5,15)(-55,15)
\psline{-}(-5,5)(-55,5)
\psline{-}(-5,-5)(-55,-5)
\psline{-}(-5,-15)(-55,-15)
\psline{-}(-5,-25)(-55,-25)
\psline{-}(-5,-35)(-55,-35)
\psline{-}(-5,-45)(-55,-45)

\psline{-}(-5,55)(-5,-45)
\psline{-}(-15,55)(-15,-45)
\psline{-}(-25,55)(-25,-45)
\psline{-}(-35,55)(-35,-45)
\psline{-}(-45,55)(-45,-45)
\psline{-}(-55,55)(-55,-45)

\pscircle*(-50,20){2}
\pscircle*(-40,-20){2}
\pscircle*(-30,-30){2}
\pscircle*(-20,50){2}
\pscircle*(-10,10){2}

\whitebox(-55,45)
\whitebox(-55,35)
\whitebox(-55,25)

\whitebox(-45,45)
\whitebox(-45,35)
\whitebox(-45,25)
\whitebox(-45,5)
\whitebox(-45,-5)
\whitebox(-45,-15)

\whitebox(-35,45)
\whitebox(-35,35)
\whitebox(-35,25)
\whitebox(-35,5)
\whitebox(-35,-5)
\whitebox(-35,-15)

\whitebox(-15,35)
\whitebox(-15,25)

\psline[linestyle=dashed,linewidth=1.5pt]{-}(-5,5)(-55,5)

\psline[linewidth=1.5pt]{-}(-35,-15)(-35,-5)(-25,-5)(-25,-15)(-35,-15)

\psline[linewidth=1.5pt]{-}(-35,25)(-35,35)(-25,35)(-25,25)(-35,25)

\rput[bl](-40,-10){\tiny{$\times$}}
\rput[bl](-30,-10){\tiny{$\times$}}
\rput[bl](-20,-10){\tiny{$\times$}}
\rput[bl](-10,-10){\tiny{$\times$}}

\rput[bl](-30,30){\tiny{$\times$}}
\rput[bl](-20,30){\tiny{$\times$}}
\rput[bl](-10,30){\tiny{$\times$}}

\rput[bl](-20,40){\tiny{$\times$}}
\rput[bl](-10,40){\tiny{$\times$}}

\rput[bl](-10,-40){\tiny{$\times$}}

\rput(-30,30){\footnotesize{$3$}}

\rput(-30,-10){\footnotesize{$2$}}

\rput(-30,-65){$w=1\;5\;\bar{4}\;\bar{3}\;2$}

\endpspicture
\caption{Type C diagrams and essential sets.\label{f.essentialC}}
\end{figure}

Just as for type B, the essential set gives a short list of rank conditions defining a symplectic degeneracy locus (or Schubert variety in $Sp_{2n}/B$).  To set up notation, let $e_{\bar{n}},\ldots,e_{\bar{1}},e_1,\ldots,e_n$ be a basis for a vector space $V$, and define a symplectic form by setting $\langle e_i,e_j\rangle = 0$ and $\langle e_{\bar{\imath}}, e_j \rangle =\delta_{i,j}$, for $i,j>0$.  The standard isotropic flag $F_\bullet$ is defined by taking $F_q$ to be the span of $e_i$, for $i\geq q$.  As before, the Schubert variety $\Omega_w$ is defined as
\[
  \Omega_w = \{ E_\bullet \,|\, \dim(E_p\cap F_q) \geq r_w(p,q) \text{ for all } 1\leq p\leq n, \; \bar{n}\leq q\leq n\}.
\]
The analogue of Corollary~\ref{c.essential-b} is this:
\begin{theorem}
The type C Schubert variety $\Omega_w\subseteq Sp_{2n}/B$ is defined by the minimal list of rank conditions $\dim(E_p\cap F_q)\geq k$, as $(k,p,q)$ ranges over $\Ess(w)$.
\end{theorem}

The statement for degeneracy loci is as follows.  Let $V$ be an even-rank vector bundle on a variety $X$, equipped with a symplectic form and two isotropic flags $V \supset E_1 \supset E_2 \supset \cdots$ and $V\supset F_1 \supset F_2 \supset \cdots$; set $F_{\bar{q}} = F_{q+1}^\perp$.  The degeneracy locus $\Omega_w \subseteq X$ associated to a signed permutation is defined by the same rank conditions as the Schubert variety.

\begin{corollary}
The rank conditions $\dim(E_p\cap F_q) \geq k$, for $(k,p,q)$ in $\Ess(w)$, suffice to define the symplectic degeneracy locus $\Omega_w\subseteq X$, and this is a minimal list of rank conditions.
\end{corollary}




\end{document}